\documentclass[11pt]{article}

\usepackage{amsfonts}
\usepackage{amsmath}

\usepackage{times}
\usepackage{bm}
\usepackage{natbib}
\usepackage{amssymb}
\usepackage{graphicx}
\usepackage[margin=2.3cm]{geometry}
 
\usepackage[plain,noend]{algorithm2e}
\usepackage{paralist}
\usepackage{amsthm}
\usepackage{authblk}
\usepackage{color}
\usepackage{hyperref}
\usepackage{comment}

\usepackage{multirow}

\hypersetup{
    colorlinks=true,
    linkcolor=blue,
    filecolor=blue,      
    urlcolor= blue,
    citecolor=blue,
}

\newtheorem{theorem}{Theorem}[section]
\newtheorem{lemma}[]{Lemma}[section]
\newtheorem{proposition}[]{Proposition}[section]
\newtheorem{remark}[]{Remark}[section]

\newtheorem{assumption}[]{Assumption}[section]

\DeclareMathOperator*{\argsup}{arg\,sup}

\begin{document}


\title{BNB autoregressions  for modeling  integer-valued time series with extreme observations\footnotemark \footnotetext{Email address: \texttt{p.gorgi@vu.nl}}}

\author[]{Paolo Gorgi}

\affil[]{Vrije Universiteit Amsterdam, The Netherlands\\ Tinbergen Institute, The Netherlands}

{\let\newpage\relax\maketitle}

\begin{abstract}
This  article introduces a general class of heavy-tailed autoregressions for modeling integer-valued time series with outliers. The proposed specification  is based on a heavy-tailed mixture of negative binomial distributions that features an observation-driven dynamic equation for the conditional expectation. The existence of a unique stationary and  ergodic  solution for the   class of  autoregressive processes is shown under a general contraction condition. The estimation of the model can be easily performed by  Maximum Likelihood given the closed form of the likelihood function. The strong consistency and the asymptotic normality of the estimator are  formally derived. Two examples of specifications  illustrate the  flexibility of the approach and the relevance of the theoretical results. In particular,  a linear dynamic equation and a score-driven  equation for the conditional expectation are considered. The score-driven  specification is shown to be particularly appealing as it delivers a robust filtering method  that attenuates the impact of outliers. An empirical application to the time series of narcotics trafficking reports in Sydney illustrates the effectiveness of the method in handling extreme observations. 
\end{abstract}



\section{Introduction} 


Time series data  with integer-valued observations are often encountered in empirical applications.  Classical continuous-response models, such as autoregressive moving average (ARMA) models, are not suited for the  modeling of such series. Over the last few decades, researchers have developed  methods that can properly account for the discreteness of the data. A standard approach is to consider observation-driven models that feature time variation in the intensity parameter of the Poisson distribution \citep{fokianos2009poisson, ferland2006integer, davis2003observation}.     
A limitation of the Poisson distribution is that it imposes  equidispersion, i.e.~mean equal to the variance. Equidispersion is  typically restrictive in empirical applications and therefore  overdispersed distributions, such as the negative binomial, are often considered \citep{davis2009negative, zhu2011negative}. Other extensions considered in the literature  include multivariate integer-valued  models  and the use of  zero-inflated distributions, which are suited for time series with large numbers of zeros. We refer the reader to \cite{davis2016handbook} for an overview of recent developments.

Extreme observations, or outliers, are often present when analyzing time series data.  The study of time series with outliers has a long history that dates back to \cite{fox1972outliers}. Ignoring extreme observations in the dataset leads  to statistical models that offer a poor description of the series of interest. Additionally, statistical inference can also be problematic in the presence of outilers.  There is a vast literature  on modeling continuous-valued time series with extreme observations. Models are typically embedded with  heavy-tailed distributions that are capable of describing  outliers as tail events.  The Student's t-distribution is often used for this purpose and robust specifications for the dynamic component of the model are  employed to attenuate the impact of outliers \citep{creal2011dynamic,HL2014}. On the other hand, to the best of our knowledge, the current literature lacks modeling methods for time series of integer-valued data when outliers are present. 


In this article, we introduce a general class of observation-driven models for integer-valued time series data with extreme observations. The approach is based on a heavy-tailed mixture of negative binomial distributions, known as the beta negative binomial (BNB) distribution.  The class of models features a dynamic location parameter and a BNB conditional distribution, which  describes extreme observations as tail events. We derive conditions for stationarity, ergodicity and finiteness  of moments for the proposed class of  stochastic processes.  Additionally, we show that   inference can be easily performed  by Maximum Likelihood (ML), given that the likelihood function is available in closed form. The strong consistency and the asymptotic normality of the ML estimator are proved under general conditions.  We consider and study two different specifications of the dynamic component of the model. The first is  a simple linear autoregression for the conditional mean. Instead, the second is based on the Generalized Autoregressive Score (GAS) framework of \cite{Creal2013} and \cite{H2013}. This second specification  delivers a robust filter  that attenuates the impact of extreme observations on the conditional expectation of the BNB process. Finally, we present an empirical analysis to the time series of police reports of narcotics trafficking in Sydney, Australia.   The results illustrate the capability  of the proposed approach in modeling time series  data with extreme observations.

The paper is structured as follows. Section \ref{section0} provides a brief review of the BNB distribution. Section \ref{section1} introduces the class of BNB autoregressive processes and discusses their stochastic properties. Section \ref{section2} derives the asymptotic properties of the ML estimator.  Section \ref{section3} introduces the linear specification of the model and discusses its properties. Section \ref{section4} introduces the score-driven specification. Section \ref{section5} presents  the empirical application.  Section  \ref{section6} concludes.

\section{Preliminaries} 
\label{section0}
We start by reviewing some properties of the BNB distribution, which will be useful in the rest of the paper.  The BNB distribution arises as a beta mixture of negative binomial distributions. In particular, let  $Y$ conditional on $P$  have a negative binomial distribution, $Y|P\sim \mathcal{NB}(r,P)$, with dispersion parameter $r$ and success probability $P$.  Assume further that $P$ has a beta distribution, $P\sim \mathcal{B}eta(\alpha,\beta)$, with shape  parameters $\alpha$ and $\beta$. Then, the marginal distribution of $Y$ is  BNB  with the following probability mass function (pmf)
\begin{align*}
\mathbb{P}(Y=y)=\frac{\Gamma(y+r)}{\Gamma(y+1)\Gamma(r)}\frac{B(\alpha+r,\beta+y)}{B(\alpha,\beta)},\quad \text{for} \quad y\in\mathbb{N},
\end{align*}
 where $\Gamma(\cdot)$ denotes the gamma function and $B(\cdot,\cdot)$ the beta function. The parameter $\alpha$ is the tail parameter of the BNB, which determines the heaviness of the right tail. The smaller $\alpha$ the heavier the tail. 
Throughout the paper, we parametrize  the BNB distribution in terms of its mean. More specifically, we consider $\beta=(\alpha-1) \lambda/r$. In this way, the parameter $\lambda$ represents the mean of the BNB distribution when the mean is finite,  which is the case when  $\alpha>1$. We say that $Y\sim \mathcal{BNB}(\lambda, r, \alpha)$ has a BNB distribution with mean $\lambda >0$, dispersion parameter $r>0$ and tail parameter $\alpha>1$ if 
\begin{align}\label{par}
\mathbb{P}(Y=y)=\frac{\Gamma(y+r)}{\Gamma(y+1)\Gamma(r)}\frac{B\Big(\alpha+r,(\alpha-1) \lambda/r+y\Big)}{B\Big(\alpha ,(\alpha-1) \lambda/r\Big)}, \quad y\in\mathbb{N}.
\end{align}
  The BNB distribution enables us to account for extreme observations, which can be seen as tail events.
Furthermore, we note that the BNB can approximate arbitrarily well the negative binomial distribution as well as the Poisson distribution.  As the  tail parameter diverges, $\alpha\rightarrow \infty$, the BNB distribution $ \mathcal{BNB}(\lambda,r,\alpha)$
converges to a negative binomial distribution with dispersion parameter $r$ and success probability $\lambda/(r+\lambda)$. Further, as the dispersion parameter diverges, $r\rightarrow \infty$, the BNB converges to a Poisson distribution with mean $\lambda$.
For a more detailed review of the BNB distribution, we refer the reader to \cite{wang2011one}.

\section{BNB autoregressive  models} 
 \label{section1}

Consider a time series of counts $\{y_t\}_{t\in \mathbb{Z}}$ with the following conditional distribution
\begin{align}
y_t|\mathcal{F}_{t-1} \sim \mathcal{BNB}(\lambda_t,r,\alpha), \quad t\in \mathbb{Z},
\label{yt_bnb}
\end{align}
where $r>0$, $\alpha>1$, and    $\mathcal{F}_{t}$  denotes the $\sigma$-field generated by $\{y_{t},y_{t-1}\dots\}$. 
The conditional mean process $\lambda_t=\mathbb{E}(y_t|\mathcal{F}_{t-1})$ is specified by the following stochastic recurrence equation (SRE)
\begin{align}
\lambda_{t+1} = g_{\pmb{\theta}}(y_t, \lambda_t), \quad t\in \mathbb{Z},
\label{bnb}
\end{align}
where $g_{\pmb{\theta}}(\cdot,\cdot)$ is a parametric updating function that maps from $\mathbb{N}\times \mathbb{R}^+$ into $\mathbb{R}^+$, and $\pmb \theta\in \mathbb{R}^{n}$ is a parameter vector. We denote by  $\pmb \kappa\in \mathbb{R}^{n+2}$   the entire parameter vector of the model $\pmb\kappa = (\pmb  \xi^\top, \pmb  \theta^\top)^\top$, where $\pmb  \xi = (r,\alpha)^\top$.  Note that in the above formulation, for simplicity of exposition, $\lambda_t$ is assumed to be $\mathcal{F}_{t-1}$-measurable. Below, we show that under a contraction condition on $g_{\pmb{\theta}}$ the model's equations admit a unique stationary and ergodic causal solution  and   $\lambda_t$ is  $\mathcal{F}_{t-1}$-measurable.


The BNB autoregressive model specified in (\ref{yt_bnb}) and  (\ref{bnb}) can describe extreme observations in time series data by means of the heavy-tailed BNB conditional pmf. A small value  of the tail parameter $\alpha$  indicates that extreme observations are more likely to occur.  The $m$-th conditional moment of the BNB autoregressive process $\mathbb{E}(y_t^m|\mathcal{F}_{t-1})$ is finite  if and only if $\alpha>m$. However, as we shall discuss below, finiteness of unconditional moments requires further conditions on the updating function $g_{\pmb{\theta}}$.  In Sections \ref{section3} and \ref{section4}, two examples of specifications of the updating function are presented. We shall see that a robust updating function may be desirable to reduce the impact of outliers on the conditional mean $\lambda_t$.

In the rest of the section, we study the stochastic properties of the BNB autoregression described by equations (\ref{yt_bnb}) and  (\ref{bnb}). We derive conditions on the updating function $g_{\pmb\theta}$ that ensure the  process to be strictly stationary and ergodic and that guarantee the existence of the first two unconditional moments of $y_t$.
 The first result that we obtain is the  stationarity and ergodicity of the process and existence of the first moment under a contraction condition. The proof of the result is based on the approach of \cite{doukhan2008weakly} for weakly dependent chains with infinite memory. We highlight that this type of contraction condition  is widely used in the literature of integer-valued processes, see , for instance,  \cite{davis2016theory} for an application to exponential families. We also obtain that the contraction condition  ensures that $\lambda_t$ is $\mathcal{F}_{t-1}$-measurable.
\begin{theorem}[stationarity and ergodicity]\label{th1}
Consider the  BNB autoregressive  process given by Equations (\ref{yt_bnb}) and  (\ref{bnb}). Furthermore, assume that the  following contraction condition holds
\begin{align}\label{contraction}
|g_{\pmb{\theta}}(y,\lambda)-g_{\pmb{\theta}}( y^*, \lambda^*)|\le c_1|y- y^*|+c_2 |\lambda-\lambda^*|,
\end{align}
where $c_1$ and $c_2$ are some positive constants such that $c_1 + c_2<1$.
Then,  the following results hold true:\\
(i) There  exists a unique strictly stationary and ergodic  causal solution $\{(y_t,\lambda_t)\}_{t\in\mathbb{Z}}$ with a finite first moment $\mathbb{E}(y_t)<\infty$.\\
(ii)  There exists a measurable function $g_{\pmb \theta}^\infty$ such that $\lambda_t=g_{\pmb \theta}^\infty(y_{t-1},y_{t-2},\dots)$, i.e.~$\lambda_t$ is $\mathcal{F}_{t-1}$-measurable.
\end{theorem}

Theorem \ref{th1} requires the sum of the Lipschitz coefficients $c_1$ and $c_2$ to be smaller than one to ensure stationarity, ergodicity and finiteness of the first unconditional moment of $y_t$.  However, a stricter contraction is needed to obtain the finiteness of the second moment. The next result imposes sufficient conditions on $c_1$ and $c_2$ to obtain  $\mathbb{E}(y_t^2)<\infty$, and hence the weak stationarity of the BNB autoregressive  process.
\begin{theorem}[weak stationarity]\label{th2} 
Assume  that $\alpha>2$ and that the contraction condition in (\ref{contraction}) holds with 
$$ \frac{(r+1)(\alpha-1)}{r(\alpha-2)} c_1^2 +c_2^2+2c_1c_2<1.$$
Then  $y_t$ has a finite second moment $\mathbb{E}(y_t^2)<\infty$. Hence $\{y_t\}_{t\in\mathbb{Z}}$ is weakly stationary.
\end{theorem}
We note that, besides a stricter contraction condition, Theorem \ref{th2} also  requires the tail parameter $\alpha$ to be greater than two. This is needed because the conditional second moment of $y_t$ is finite if only if $\alpha>2$.

The results in Theorems \ref{th1} and \ref{th2} are of key importance to derive the asymptotic properties of the ML estimator. In particular, as we shall see in Section \ref{section2}, the strict stationarity condition  in Theorem  \ref{th1} is sufficient for the consistency of the ML estimator. Instead, the additional conditions in Theorem \ref{th2} are needed for the asymptotic normality to hold. In Sections \ref{section3} and \ref{section4}, Theorems \ref{th1} and \ref{th2}  will be employed to establish the stochastic properties of a linear  and a score-driven BNB autoregression.

\section{Maximum Likelihood estimation} 
\label{section2}
In this section, we discuss the estimation of the BNB autoregression by ML. We derive conditions to ensure  consistency and  asymptotic normality of the ML estimator.
 We assume that a subset of a realized path from the BNB autoregressive process in (\ref{yt_bnb}) and (\ref{bnb}) with true parameter value $\pmb\kappa=\pmb\kappa_0$ is observed $\{y_t\}_{t=1}^T$. Here $T$ denotes the sample size.  The likelihood function is available in closed form through a prediction error decomposition. In particular, the average log-likelihood function  is
\begin{align*}
\hat L_T({\pmb{\kappa}})=\frac{1}{T}\sum_{t=1}^T \hat l_t({\pmb{\kappa}}) =\frac{1}{T}\sum_{t=1}^T \log p(y_t|\hat \lambda_t({\pmb{\theta}}), r, \alpha),
\end{align*}
where $p$ denotes the conditional pmf, given by
\begin{align*}
p(y_t|\hat \lambda_t({\pmb{\theta}}), r, \alpha)=\frac{\Gamma(y_t+r)}{\Gamma(y_t+1)\Gamma(r)}\frac{B\Big(\alpha+r,(\alpha-1) \hat \lambda_t({\pmb{\theta}})/r+y_t\Big)}{B\Big(\alpha ,(\alpha-1) \hat \lambda_t({\pmb{\theta}})/r\Big)}.
\end{align*}
The filtered time-varying parameter $\hat \lambda_t({\pmb{\theta}})$ is obtained recursively using the observed data $\{y_t\}_{t=1}^T$
\begin{align}
 \hat \lambda_{t+1}({\pmb{\theta}}) = g_{\pmb{\theta}}(y_t, \hat \lambda_t({\pmb{\theta}})),
  \label{filter}
 \end{align}
 where the recursion is initialized at a fixed point $\hat \lambda_{1}({\pmb{\theta}}) \in \mathbb{R}^+$. We note that  initializing the recursion in  (\ref{filter}) is needed since the observed data starts from time $t=1$. This is quite standard in the literature of observation-driven models.  Finally, the ML estimator $\hat {\pmb{\kappa}}_T$ is  defined as the maximizer of the likelihood function 
 \begin{align}
\hat {\pmb{\kappa}}_T =\argsup_{{\pmb{\kappa}} \in  {K}} \hat L_T({\pmb{\kappa}}),
\label{MLE}
\end{align}
where $K =\Xi \times\Theta$ with $\Xi \subset (0,\infty)\times (2,\infty)$ and $\Theta \subset \mathbb{R}^n$ being compact parameter sets.

\subsection{Consistency} 
In order to establish the consistency of the ML estimator, we first derive the stochastic limit properties of the filtered parameter $\hat \lambda_t$ defined in (\ref{filter}). Note that $\hat \lambda_t(\cdot)$ is a stochastic function that maps from $\Theta$ into $\mathbb{R}^+$.  The stability of the filtered parameter $\hat \lambda_t$ is often referred in the literature as invertibility \citep{SM2006,blasques2018feasible}. Because of the initialization, $\hat \lambda_t$ evaluated at the true parameter value, $\hat\lambda_t(\pmb\theta_0)$, does not correspond to the true conditional mean $\lambda_t$.  In the following, we show that $\{\hat\lambda_t\}_{t\in\mathbb{N}}$ converges exponentially a.s.~(e.a.s.)\footnote{A sequence of  random variables $\{\hat x_t\}_{t\in\mathbb{N}}$  converges e.a.s.~to another sequence $\{\tilde x_t\}_{t\in\mathbb{N}}$  if there is a constant $c>1$ such that $c^t |\hat x_t-\tilde x_t| \xrightarrow{a.s.}$ 0 as $t\rightarrow\infty$.}~and uniformly in $\Theta$ to a stationary and ergodic sequence of functions $\{\tilde\lambda_t\}_{t\in\mathbb{N}}$ such that $\tilde\lambda_t(\pmb\theta_0)=\lambda_t$ with probability one. We start by imposing a continuity condition on the updating function $g_{\pmb{\theta}}$, which ensures that $\pmb \theta \mapsto \hat \lambda_t({\pmb{\theta}})$ is continuous in $\Theta$ with probability one.
\begin{assumption}\label{a1}
The function  $({\pmb{\theta}},\lambda) \mapsto g_{\pmb{\theta}}(y, \lambda)$ is continuous in $\Theta \times \mathbb{R}^+ $ for any $y\in\mathbb{N}$.
\end{assumption}
Next, we assume that the contraction condition holds for all $\pmb{\theta}$ in the parameter set $\Theta$, which contains the true parameter value $\pmb{\theta}_0$. We note that this assumption is not restrictive since, in general,  $\Theta$ can be defined as a compact ball around the true parameter vector $\pmb{\theta}_0$.
\begin{assumption}\label{a2}
The contraction condition in (\ref{contraction}) is satisfied for any ${\pmb{\theta}} \in \Theta$. Furthermore $\pmb{\theta}_0\in\Theta$.
\end{assumption}
The next result ensures the uniform convergence over $\Theta$ of the filtered parameter $\hat \lambda_t({\pmb{\theta}})$ to a stationary and ergodic limit $\tilde \lambda_t({\pmb{\theta}})$. Here $\|\cdot\|_{\Theta}$ denotes the supremum norm. Given a function $f: \Theta \mapsto \mathbb{R}$, the supremum norm is  $\|f\|_{\Theta}=\sup_{\pmb\theta \in\Theta} |f(\pmb\theta)|$. 

\begin{proposition}[invertibility]
Let Assumptions \ref{a1} and \ref{a2} hold, then the filter $\{\hat \lambda_t({\pmb{\theta}})\}_{t \in \mathbb{N}}$ converses e.a.s.~and  uniformly over $\Theta$ to a unique stationary and ergodic sequence $\{\tilde \lambda_t({\pmb{\theta}})\}_{t \in \mathbb{Z}}$,
$$\|\hat \lambda_t - \tilde \lambda_t \|_\Theta \xrightarrow{e.a.s.}0, \quad t\rightarrow \infty,$$
for any initialization $\hat \lambda_1({\pmb{\theta}}) \in \mathbb{R}^+$. Furthermore, $\tilde \lambda_t$ has a bounded moment uniformly over $\Theta$, $\mathbb{E}\|\tilde\lambda_t\|_{\Theta}<\infty$.
\label{invertibility}
\end{proposition}

Proposition \ref{invertibility} plays a crucial role to ensure that the  likelihood function $\hat l_t(\pmb\tau)$, which depends on  the approximate filter $\hat \lambda_t$, converges to a stationary and ergodic limit  $l_t(\pmb\tau)=p(y_t| \tilde\lambda_t({\pmb{\theta}}), r, \alpha)$, which depends on limit filter $\tilde\lambda_t$. In this way, $l_t(\pmb\tau_0)$ corresponds to the true conditional log-pmf of the BNB autoregressive model. 

Next, building on the invertibility result, we impose some additional conditions to obtain the strong consistency of the ML estimator. The next assumption imposes  a lower bound on the updating function.
\begin{assumption}[lower bound]\label{a3}
There is a constant  $\bar c>0$ such that $g_{\pmb{\theta}}(y,\lambda)\ge \bar c$ for any $({\pmb{\theta}},y,\lambda)\in \Theta\times \mathbb{N}\times \mathbb{R}^+$.
\end{assumption}

Finally, we impose an identifiability condition on the parametric updating function $g_{{\pmb{\theta}}}$. This condition is needed to ensure that different parameter values of ${\pmb{\theta}}$ give observationally different paths of $\tilde\lambda_t({\pmb \theta})$.
\begin{assumption}[identifiability]\label{a4}
For any ${\pmb{\theta}}_1,{\pmb{\theta}}_2\in\Theta$ and $\lambda \in \mathbb{R}^+$, the equality $g_{{\pmb{\theta}}_1}(y,\lambda)=g_{{\pmb{\theta}}_2}(y,\lambda)$ holds true for all $y \in \mathbb{N}$ if and only if ${\pmb{\theta}}_1={\pmb{\theta}}_2$.
\end{assumption}

Under these conditions, we obtain the strong consistency of the ML estimator.
\begin{theorem}[consistency]\label{th3}
Let Assumptions  \ref{a1}-\ref{a4} hold and let ${\pmb{\kappa}}_0\in K$, then the ML estimator defined in (\ref{MLE}) is strongly consistent, that is,
$$\hat {\pmb{\kappa}}_T\xrightarrow{a.s.}{\pmb{\kappa}}_0, \quad \text{as}  \quad  T\rightarrow \infty.$$
\end{theorem}

\subsection{Asymptotic normality} 

We now focus on deriving the asymptotic normality of the ML estimator. First, we require the data generating process to have a finite second moment. A finite second moment  is needed to ensure the existence of some moments for the derivatives of the log-likelihood, which are used to apply a central limit theorem to the score of the log-likelihood function.
\begin{assumption}[weak stationarity]\label{a5.0}
The assumptions of Theorem \ref{th2} are satisfied for $\pmb \kappa = \pmb{\kappa}_0$.
\end{assumption}
The next assumption is needed to ensure that the Fisher information matrix is positive definite.
\begin{assumption}[positive definite Fisher information]\label{a5}
The random variables of the vector $\frac{\partial g_{\pmb{\theta}}}{\partial {\pmb{\theta}}}(y_t,\lambda_t)|_{{\pmb{\theta}}={\pmb{\theta}}_0}$ are linearly independent.
\end{assumption}

Finally, the next assumption requires some regularity conditions on the updating function $g_{\pmb\theta}$. In particular, we impose the updating function to be twice continuously differentiable and have some of its derivatives bounded by some linear functions of their arguments. Note that $\|\cdot\|$ denotes the $L_1$-norm when applied to a vector and the  operator norm induced by the  $L_1$-norm when applied to a matrix. Furthermore, we consider the following shorthand notation for the derivatives of the updating function: $g_{\pmb{\theta}}^{d}(y,\lambda)=\partial g_{\pmb{\theta}}(y,\lambda)/ \partial d$. If $d$ is a 2-dimensional vector, then $g_{\pmb{\theta}}^{d}(y,\lambda)$ denotes the second order partial derivative with respect to the elements of the vector. For instance, if $d=(\pmb \theta, \pmb \theta)$, then $g_{\pmb{\theta}}^{\pmb \theta  \pmb \theta}(y,\lambda)=\partial^2 g_{\pmb{\theta}}(y,\lambda)/ \partial \pmb \theta \partial \pmb \theta^\top$.

\begin{assumption}\label{a6}
The function $(\pmb{\theta}, \lambda) \mapsto g_{\pmb{\theta}}(y,\lambda)$ is twice continuously differentiable with respect to both $\pmb{\theta}$ and $\lambda$.  Furthermore,  for any ${\pmb{\theta}} \in \Theta$
$$\|g_{\pmb{\theta}}^{d}(y,\lambda)-g_{\pmb{\theta}}^{d}(y^*,\lambda^*)\| \le b_1|y-y^*|+b_2|\lambda-\lambda^*|,$$
for $d\in\{\lambda,{\pmb{\theta}},(\lambda,\lambda),({\pmb{\theta}},\lambda),({\pmb{\theta}},{\pmb{\theta}})\}$ and some positive constants  $b_1$ and $b_2$. 
\end{assumption}

The next result delivers the asymptotic normality of the ML estimator.

\begin{theorem}[asymptotic normality]\label{th4}
Let Assumptions   \ref{a1}-\ref{a6} hold and ${\pmb{\kappa}}_0\in \text{int}(K)$, then the ML estimator defined in (\ref{MLE}) has the following asymptotic distribution
$$\sqrt{T}(\hat {\pmb{\kappa}}_T-{\pmb{\kappa}}_0)\xrightarrow{d}N(\pmb{0},{\pmb{\text F}}_0^{-1}), \quad \text{as}  \quad  T\rightarrow \infty,$$
where ${{\pmb{\text F}}}_0=-\mathbb{E}\left[\frac{\partial^2l_t(\pmb{\kappa}_0)}{\partial {\pmb{\kappa}}\partial {\pmb{\kappa}}^\top}\right]$.
\end{theorem}

In practice, the Fisher information matrix ${\pmb{\text F}}_0$ needs to be estimated. A consistent estimator is obtained by plugging in the ML estimator into the second derivative of the log-likelihood 
\begin{align*}
\hat{\pmb{\text F}}_T=-\frac{1}{T} \sum_{t=1}^T \frac{\partial^2 \hat l_t(\hat{\pmb{\kappa}}_T)}{\partial {\pmb{\kappa}}\partial {\pmb{\kappa}}^\top}.
\end{align*}
The next result shows that the estimator given above delivers a strongly consistent estimate for the asymptotic covariance matrix of the ML estimator.
\begin{proposition}\label{pr2}
Let the assumptions of Theorem \ref{th4} hold, then  the estimator of the asymptotic covariance matrix is strongly consistent, that is,
\begin{align*}
{\hat{\pmb{\text F}}^{-1}}_{T}\xrightarrow{a.s.} {\pmb{\text F}}_0^{-1}, \quad \text{as} \quad T\rightarrow \infty.
\end{align*}
\end{proposition}

\section{BNB-INGARCH model} 
\label{section3}
\subsection{The model} 
An intuitive and simple way to specify the conditional mean $\lambda_t$ is to consider a linear autoregressive process driven by past observations.  Count processes with a linear specification  of the conditional mean are often referred in the literature as integer-valued generalized autoregressive conditional heteroscedastic (INGARCH) models \citep{ferland2006integer}. We define the BNB-INGARCH model through the following equations
\begin{align}
y_t|\mathcal{F}_{t-1} \sim \mathcal{BNB}(\lambda_t,r,\alpha),\quad \lambda_{t+1} = \omega  + \phi  \lambda_{t} + \tau y_t,
\label{linear_bnb}
\end{align}
where $\omega>0$, $\phi\ge0$ and $\tau>0$ are static parameters to be estimated. These parameters are restricted  to be positive to guarantee that $\lambda_t$ is strictly positive with probability one.  A linear autoregressive specification for the conditional expectation $\lambda_t$ is often considered for  Poisson and  negative binomial autoregressions, see \cite{ferland2006integer}, \cite{fokianos2009poisson} and  \cite{zhu2011negative}. We refer to these models as the Po-INGARCH   and NB-INGARCH models, respectively. We can immediately see that the BNB-INGARCH model  can approximate arbitrarily well both the Po-INGARCH   and the NB-INGARCH model since the BNB distribution converges to the negative binomial as $\alpha\rightarrow\infty$ and to the Poisson as, additionally, $r\rightarrow\infty$. 

The next result relies on Theorem \ref{th1} and \ref{th2} to derive conditions for strict and weak stationarity of the BNB-INGARCH process.
\begin{theorem}\label{p1}
Let the BNB-INGARCH process in  (\ref{linear_bnb}) satisfy
\begin{align}\label{cs1}
\tau +\phi<1.
\end{align}
Then, the process admits a strictly stationary and ergodic solution with a finite first moment $\mathbb{E}(y_t)<\infty$. 
Additionally,  let $\alpha>2$ and 
\begin{align}\label{cs2}
 \frac{(r+1)(\alpha-1)}{r(\alpha-2)} \tau^2 +\phi^2+2\tau \phi<1.
\end{align}
Then, the stationary solution has a finite second moment $\mathbb{E}(y_t^2)<\infty$. Hence, it is weakly stationary.
\end{theorem}
Finally, we derive the strong consistency and asymptotic normality of the ML estimator of the BNB-INGARCH model by appealing to Theorem \ref{th3} and \ref{th4}.
\begin{theorem}\label{th:ML}
Let the observed series $\{y_t\}_{t=1}^T$ be generated by the  BNB-INGARCH process in (\ref{linear_bnb}) with parameter value $\pmb \kappa_0=(r_0,\alpha_0,\omega_0,\phi_0,\tau_0)^\top$. Furthermore, let $\pmb \kappa_0 \in K$ where,  $K$ is compact parameter set  such that $\phi+\tau<1$ and $\omega>0$ for any $\pmb \kappa \in K$. Then,  the  ML estimator defined in  (\ref{MLE})  is strongly consistent
$$\hat {\pmb{\kappa}}_T\xrightarrow{a.s.}{\pmb{\kappa}}_0, \quad \text{as}  \quad  T\rightarrow \infty.$$
Assume further that $\pmb \kappa_0$ satisfies the contraction condition in (\ref{cs2}), $\alpha_0>2$, and $\pmb \kappa_0 \in int(K)$. Then, the ML estimator is asymptotically normally distributed
$$\sqrt{T}(\hat {\pmb{\kappa}}_T-{\pmb{\kappa}}_0)\xrightarrow{d}N(\pmb{0},{\pmb{\text F}}_0^{-1}), \quad \text{as}  \quad  T\rightarrow \infty,$$
where ${{\pmb{\text F}}}_0=-\mathbb{E}\left[\frac{\partial^2l_t(\pmb{\kappa}_0)}{\partial {\pmb{\kappa}}\partial {\pmb{\kappa}}^\top}\right]$.
\end{theorem}
The proof of Theorem \ref{th:ML}, which is given in the Appendix,  is obtained by checking that Assumptions \ref{a1}-\ref{a6} are satisfied. The small sample properties of the ML estimator are studied in a Monte Carlo experiment in the next section.


\subsection{Monte Carlo simulation study} 
We investigate the small-sample properties of the ML estimator by means of a Monte Carlo simulation experiment. 
We generate samples of different sizes from the BNB-INGARCH model in (\ref{linear_bnb}) for several parameter values. The parameters are then estimated by ML. Table \ref{Tab1} reports the results of the experiment. The     parameters    $r$ and $\alpha$  are reparameterized in terms of their inverse.  This is done because, especially in small samples, a given realized path from the BNB-INGARCH process may not present outliers and   the estimate of $\alpha$ may become arbitrarily large since the likelihood function is flat for large   $\alpha$. 
\begin{table}[h!]

\centering

\caption{{\small Simulation results  for the ML estimator obtained from 1,000 Monte Carlo replications. The mean, the standard deviation (SD) and  the root mean square error (RMSE) are reported for different parameter values and sample sizes. The parameter $\delta$ represents the unconditional mean of $y_t$, i.e.~$\delta=\omega/(1-\phi-\tau)$.}}

\resizebox{0.95\columnwidth}{!}{
\begin{tabular}{lrccccccccccc}
  \hline
  \hline
\vspace{-0.7cm}\\ 
 && $\delta$ & $\phi$ & $\tau$ & $r^{-1}$ & $\alpha^{-1}$ &  & $\delta$ & $\phi$ & $\tau$ & $r^{-1}$ & $\alpha^{-1}$ \\ 
  \hline
\vspace{-0.7cm}\\ 
\multicolumn{2}{c}{\textbf{True value}} & \textbf{10.00} & \textbf{0.50} &  \textbf{0.20} &  \textbf{0.10} &  \textbf{0.20} &  &  \textbf{10.00} &  \textbf{0.68} &  \textbf{0.20} &  \textbf{0.10} &  \textbf{0.20} \vspace{0.05cm}\\ 
\multirow{3}{*}{$T=250$} & Mean & 9.965 & 0.457 & 0.198 & 0.125 & 0.172 &  & 9.991 & 0.648 & 0.197 & 0.116 & 0.167\vspace{-0.15cm} \\ 
 & SD & 0.984 & 0.191 & 0.065 & 0.091 & 0.052 &  & 1.607 & 0.120 & 0.059 & 0.093 & 0.052 \vspace{-0.15cm} \\  
 & RMSE & 0.985 & 0.196 & 0.065 & 0.094 & 0.059 &  & 1.607 & 0.124 & 0.059 & 0.094 & 0.062 \vspace{0.2cm}\\ 
\multirow{3}{*}{$T=500$} & Mean & 9.960 & 0.476 & 0.199 & 0.124 & 0.183 &  & 9.984 & 0.664 & 0.199 & 0.114 & 0.182 \vspace{-0.15cm} \\ 
 & SD & 0.689 & 0.135 & 0.046 & 0.079 & 0.037 &  & 1.192 & 0.071 & 0.040 & 0.073 & 0.037\vspace{-0.15cm} \\  
 &RMSE  & 0.690 & 0.137 & 0.046 & 0.082 & 0.040 &  & 1.192 & 0.073 & 0.040 & 0.075 & 0.041 \vspace{0.2cm}\\  
\multirow{3}{*}{$T=1000$} & Mean & 9.976 & 0.492 & 0.199 & 0.118 & 0.190 &  & 9.997 & 0.672 & 0.199 & 0.113 & 0.190 \vspace{-0.15cm} \\ 
 &SD  & 0.509 & 0.091 & 0.033 & 0.065 & 0.026 &  & 0.853 & 0.048 & 0.029 & 0.062 & 0.026\vspace{-0.15cm} \\ 
&  RMSE& 0.510 & 0.091 & 0.033 & 0.067 & 0.027 &  & 0.853 & 0.048 & 0.029 & 0.063 & 0.028 \vspace{0.2cm}\\ 
\multirow{3}{*}{$T=2000$} & Mean & 9.982 & 0.497 & 0.199 & 0.113 & 0.194 &  & 10.013 & 0.677 & 0.199 & 0.108 & 0.196 \vspace{-0.15cm} \\ 
 &SD  & 0.363 & 0.060 & 0.023 & 0.053 & 0.019 &  & 0.592 & 0.033 & 0.020 & 0.044 & 0.019\vspace{-0.15cm} \\ 
 & RMSE & 0.364 & 0.060 & 0.023 & 0.055 & 0.020 &  & 0.592 & 0.033 & 0.020 & 0.044 & 0.019 \\ 
\hline
  \hline
\vspace{-0.7cm}\\ 
\multicolumn{2}{c}{\textbf{True value}} & \textbf{10.00} & \textbf{0.50} &  \textbf{0.20} &  \textbf{0.10} &  \textbf{0.10} &  &  \textbf{10.00} &  \textbf{0.68} &  \textbf{0.20} &  \textbf{0.10} &  \textbf{0.10} \vspace{0.05cm}\\ 
\multirow{3}{*}{$T=250$} & Mean & 9.986 & 0.449 & 0.198 & 0.108 & 0.072 &    & 10.042 & 0.647 & 0.197 & 0.112 & 0.073 \vspace{-0.15cm} \\  
   & SD & 0.744 & 0.193 & 0.066 & 0.079 & 0.033 &   & 1.170 & 0.119 & 0.056 & 0.083 & 0.033\vspace{-0.15cm} \\  
 &RMSE  & 0.744 & 0.199 & 0.066 & 0.080 & 0.043 &  & 1.171 & 0.123 & 0.056 & 0.084 & 0.043 \vspace{0.2cm}\\  
\multirow{3}{*}{$T=500$}& Mean & 9.971 & 0.474 & 0.198 & 0.107 & 0.081 &   & 9.963 & 0.665 & 0.197 & 0.107 & 0.081 \vspace{-0.15cm} \\ 
  & SD & 0.533 & 0.134 & 0.046 & 0.066 & 0.024 &    & 0.821 & 0.071 & 0.040 & 0.069 & 0.024 \vspace{-0.15cm} \\ 
   & RMSE & 0.534 & 0.137 & 0.046 & 0.067 & 0.031 &    & 0.822 & 0.073 & 0.040 & 0.069 & 0.031 \vspace{0.2cm}\\ 
\multirow{3}{*}{$T=1000$} & Mean & 9.994 & 0.490 & 0.198 & 0.107 & 0.087 &    & 10.012 & 0.673 & 0.198 & 0.108 & 0.088 \vspace{-0.15cm} \\ 
   & SD & 0.384 & 0.093 & 0.031 & 0.057 & 0.017 &  & 0.608 & 0.049 & 0.027 & 0.056 & 0.017 \vspace{-0.15cm} \\ 
   & RMSE & 0.384 & 0.093 & 0.031 & 0.057 & 0.021 &    & 0.608 & 0.050 & 0.027 & 0.057 & 0.021 \vspace{0.2cm}\\ 
\multirow{3}{*}{$T=2000$} & Mean & 9.999 & 0.497 & 0.199 & 0.105 & 0.092 &   & 9.994 & 0.676 & 0.199 & 0.105 & 0.093 \vspace{-0.15cm} \\ 
   & SD & 0.271 & 0.065 & 0.023 & 0.046 & 0.012 &   & 0.434 & 0.035 & 0.020 & 0.043 & 0.011 \vspace{-0.15cm} \\  
   & RMSE & 0.271 & 0.065 & 0.023 & 0.046 & 0.015 &   & 0.434 & 0.035 & 0.020 & 0.043 & 0.013 \\ 
   \hline
\end{tabular}
}
\label{Tab1}
\end{table}

The results show that the estimates of all parameters seem to be consistent since the RMSE decreases as the sample size increases.   Furthermore, we note that the  estimation bias for   $\delta$, $\delta=\omega/(1-\phi-\tau)$, and $\tau$ is negligible even for the smallest sample size considered in the experiment ($T=250$). This can be noted from the fact that the standard deviation is equivalent to the RMSE. On the other hand, the results show the presence of a small-sample bias in the estimates of $\phi$, $r$ and $\alpha$. In particular, the parameters $\phi$ and $r$ tend to be underestimated  and $\alpha$ tends to be overestimated. However, the bias of all these estimates becomes negligible for larger sample sizes ($T=1,000$ and $T=2,000$). Finally, we note that the results seem to be coherent across the different parameter values. This is true even when  $\phi=0.68$, which represents  scenarios  where the  parameters are close to the boundary of the  weak stationarity region given in   (\ref{cs2}).

\section{A score-driven BNB autoregression} 
\label{section4}

The BNB-INGARCH model presented in Section \ref{section3} accounts for extreme observations by means of the heavy tail  of the conditional BNB distribution. However, the SRE for $\lambda_t$ is not robust against outliers: an extreme value of $y_t$ can have an arbitrary large impact on $\lambda_{t+1}$. In practice, it is often desirable to have a robust SRE that attenuates the impact of extreme observations on  $\lambda_t$. We propose a robust specification that is based on the GAS framework of \cite{Creal2013} and \cite{H2013}. The conditional mean $\lambda_t$ is specified as an autoregressive process with innovation given by the score of the predictive likelihood. Score-driven models are widely used in the literature to specify robust updating functions in models with heavy-tailed distributions \citep{HL2014,opschoor2017new}.

The BNB autoregression  with  a score-driven SRE, which we label as the BNB-GAS, is specified by the following equations
\begin{align}\label{BNB-GAS}
y_t|\mathcal{F}_{t-1} \sim \mathcal{BNB}(\lambda_t,r,\alpha),\quad \log  \lambda_{t+1} = \omega  + \phi  \log \lambda_{t} + \tau s_t,
\end{align}
where $\omega \in\mathbb{R}$, $\phi \ge 0$ and $\tau>0$.
The score innovation $s_t=\partial \log p(y_t|\lambda_t,r,\alpha)/\partial \log \lambda_t$ is given by 
\begin{align}\label{score}
s_t = \gamma\lambda_t \Big(\psi(\gamma\lambda_t+y_t)+\psi(\gamma\lambda_t+\alpha)-\psi(\gamma\lambda_t +y_t +\alpha +r)-\psi(\gamma\lambda_t)\Big),
\end{align}
where $\gamma=(\alpha-1)/r$  and $\psi(\cdot)$ denotes the digamma function, i.e.~$\psi(x)=\partial \log\Gamma(x)/\partial x$. The exponential link function in (\ref{BNB-GAS}) is considered to ensure that $\lambda_t$ is strictly positive with probability one. We refer to the score   $s_t$   as the innovation of the process since $\mathbb{E}(s_t|\mathcal{F}_{t-1})=0$ with probability one.

The peculiarity of the BNB-GAS model is that the functional form of the score innovation reduces the impact of outliers. Figure \ref{score_impact} illustrates  the  sensitivity  of $s_t$ to $y_t$ for different values of the tail parameter $\alpha$. We can see that the effect of large values of $y_t$ on $s_t$ is attenuated and the degree of  attenuation depends on the tail parameter $\alpha$. The smaller the parameter $\alpha$  the more robust the score innovation $s_t$.  This  behavior of $s_t$ is quite intuitive since a small  $\alpha$ introduces heavy tails in the conditional pmf of $y_t$ and therefore it generates outliers in the observed time series, see also \cite{HL2014} for a similar interpretation in the context of Student's t-distributions.  Furthermore, as it is shown in the proof of Theorem \ref{SE_BNBINGAS} in the Appendix, the score innovation is bounded by a constant $|s_t|\le (\alpha+r+1)$. Therefore, $s_t$ is robust since it does not go to infinity as $y_t\rightarrow \infty$.

\begin{figure}[h!]
\center
\includegraphics[scale=0.65]{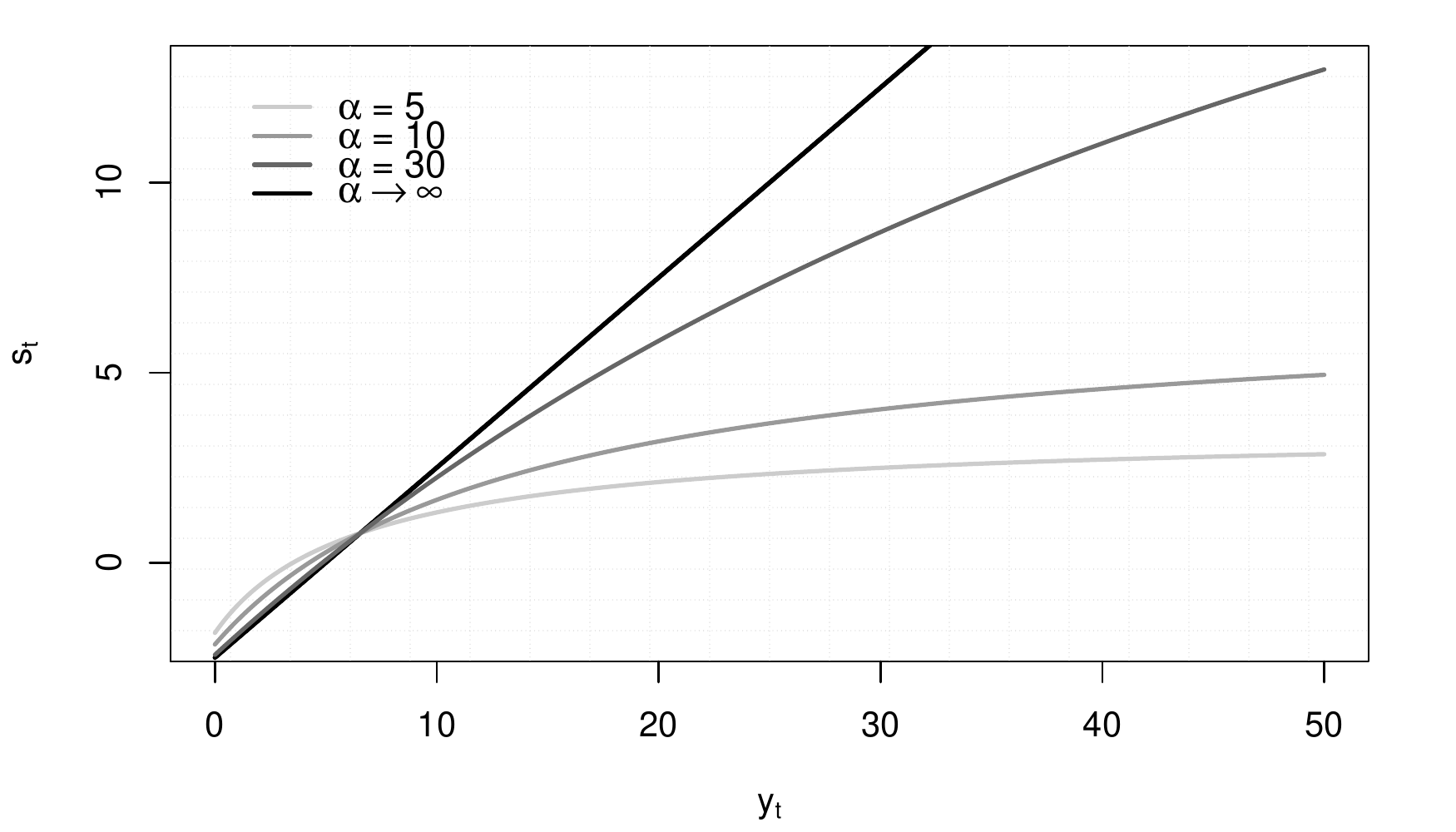}
\vspace{-0.5cm}
\caption{\small{Impact of $y_t$ on the score innovation $s_t$ for different values of the tail parameter $\alpha$.}}
\label{score_impact}
\end{figure} 

The BNB-GAS can approximate arbitrarily well some existing models that have been proposed in the literature.  As $\alpha\rightarrow \infty$, the BNB-GAS becomes a score-driven model with  negative binomial distribution, see \cite{gorgi2018integer} for an application of the GAS framework with negative binomials. As additionally $r\rightarrow \infty$, the  model  becomes a Poisson autoregressive model, which belongs to the class of models introduced by \cite{davis2003observation}.

We now focus on the stochastic properties of the BNB-GAS. The next theorem gives sufficient conditions for the existence of a stationary and ergodic solution for the BNB-GAS process.
\begin{theorem}\label{SE_BNBINGAS}
Consider the BNB-GAS process given in (\ref{BNB-GAS}) and (\ref{score}) and let $0\le \phi<1$ and 
\begin{align}\label{SEGAS}
\left(\tau \frac{\gamma(3\alpha+2r+3)+(\alpha+r+1)}{\gamma}+\phi \right) \exp\big(\tau (2\alpha+r+2)\big)<1.
\end{align}
Then, the process admits a stationary and ergodic solution with finite first moment $\mathbb{E}(y_t)<\infty$.
\end{theorem}
The proof of this theorem is obtained by an application of Theorem \ref{th1}. We note that the parameter restriction in (\ref{SEGAS}) is  a sufficient condition  that makes the contraction condition of Theorem \ref{th1} hold. Given the complex functional form of $s_t$, it is not straightforward  to obtain sharper upper bounds for the  Lipschitz coefficients $c_1$ and $c_2$ in  (\ref{contraction}). The condition imposed on the parameters by (\ref{SEGAS}) may be restrictive in practice, however, this condition highlights that the stationarity region is not degenerate.

\begin{remark}\label{RE_BNBINGAS}
Under the conditions of Theorem \ref{th1}, the BNB-GAS process admits a stationary and ergodic solution with  $\mathbb{E}(y_t^m)<\infty$ if and only if $\alpha>m$. This result holds  true because $\lambda_t$ takes vales on a compact set with probability one, see the proof of Theorem \ref{th1} in the Appendix. This means that $\mathbb{E}(\lambda_t^m)<\infty$ for any $m\in \mathbb{N}$. Therefore $\mathbb{E}(y_t^m)<\infty$ if and only if  $\mathbb{E}(y_t^m|\mathcal{F}_{t-1})<\infty$ with probability one, namely,  $\alpha>m$.
\end{remark}


\section{Empirical application} 
\label{section5}

In this section, we present an empirical application to the monthly  number  of police reports on narcotics trafficking  in Sydney, Australia.  The time series is from January 1995 to December 2016 and it is available in the New South Wales dataset of police reports.
Figure \ref{figts} displays the plot  and the empirical autocorrelation functions of the series. We can see that the dataset presents some extreme observations. In particular, the number of narcotics trafficking reports is exceptionally high in  August 2000, March 2008 and May 2015. Therefore,  BNB autoregressive models seem particularly suited to describe the autocorrelation structure  and account for the outliers in the dataset.

\begin{figure}[h!]
\center
\includegraphics[scale=0.65]{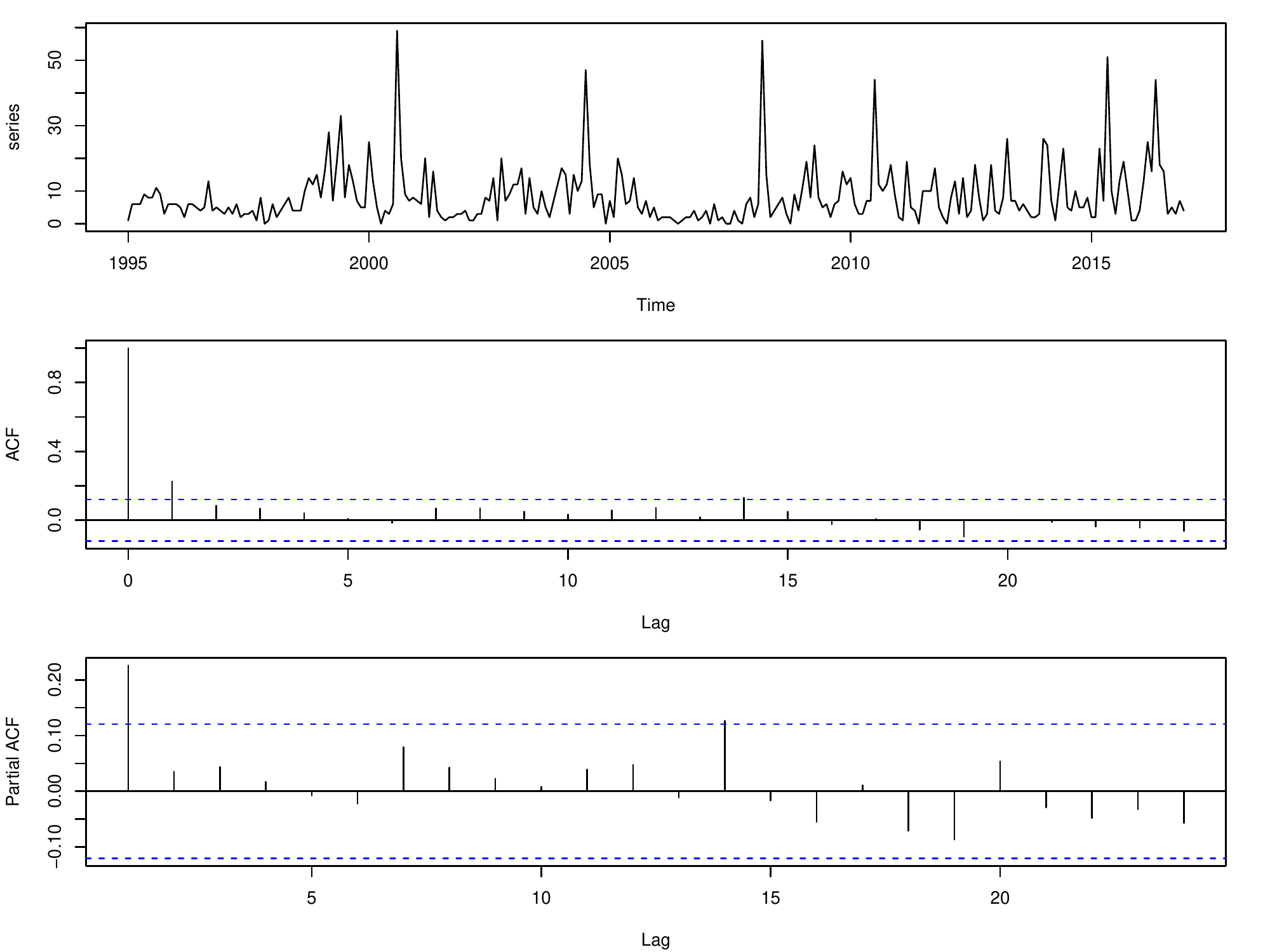}
\vspace{-0.5cm}
\caption{\small{The first plot shows the monthly number of police reports for narcotics trafficking  in Sydney. The second and third plots present the sample autocorrelation functions of the series.}}
\label{figts}
\end{figure}

Besides the BNB-INGARCH and the BNB-GAS, we consider two negative binomial specifications: one with a linear updating function  and  one based on the GAS framework, which we label as linear NB-INGARCH and NB-GAS respectively. As discussed before, these two models are limit cases, $\alpha\rightarrow\infty$, of the   BNB-INGARCH and the BNB-GAS model. Table \ref{tab2} reports the estimation results. The BNB specifications give a better description of the time series since they have lower values of the Akaike information criterion (AIC). More specifically, the BNB-GAS is the model that best fits the data. This suggests that the robust updating function given by the score innovation is beneficial in this case.
Furthermore, the relevance of the BNB distribution can also be elicited from the relatively low estimates of the tail parameter $\alpha$, which is estimated to be around $5$ with a standard error of about $0.8$ for both BNB specifications. This further  indicates that the extreme observations in the data   are not properly described by a negative binomial distribution.

\begin{table}[ht]
\centering
\caption{\small{Maximum Likelihood estimates of the models. Standard errors are in brackets. The last two columns contain the log-likelihood   and the AIC, respectively. The parameter $\delta$ is $\delta=\omega/(1-\phi-\tau)$ for the INGARCH models and $\delta=\omega/(1-\phi)$ for the GAS models.}}
\begin{tabular}{lccccccc}
 \cline{1-8}
\vspace{-0.7cm} \\
 & $\delta$ & $\phi$ & $\tau$ & $r $ & $\alpha $  & log-lik  & AIC \\ 
 \cline{1-8}\vspace{-0.7cm} \\
BNB-INGARCH  & 8.549 & 0.481  & 0.267 & 6.521 & 4.819& -807.66 &  1625.33 \vspace{-0.3cm} \\ 
  & \footnotesize{(1.104)} & \footnotesize{(0.223)} & \footnotesize{(0.085)} & \footnotesize{(3.394)} &\footnotesize{(0.744)}&  &  \vspace{0.1cm}\\ 
BNB-GAS  & 2.087 & 0.714 & 0.197 & 4.408   & 5.029 & -807.04 & \textbf{1624.09}  \vspace{-0.3cm} \\  
  & \footnotesize{(0.107)} &\footnotesize{(0.169)}  &\footnotesize{(0.056)}  & \footnotesize{(1.923)}  &\footnotesize{(0.849)}&  &  \vspace{0.1cm}\\ 
NB-INGARCH  & 8.675 & 0.307 & 0.308 & 1.561 & - & -821.22 & 1650.45 \vspace{-0.3cm} \\ 
  & \footnotesize{(0.907)} & \footnotesize{(0.254)} & \footnotesize{(0.099)} & \footnotesize{(0.159)} && &  \vspace{0.1cm}\\ 
NB-GAS   & 2.102& 0.699  & 0.140 &   1.540 & - & -822.59 & 1653.19  \vspace{-0.3cm} \\  
  & \footnotesize{(0.087)} &\footnotesize{(0.232)}  & \footnotesize{(0.051)} & \footnotesize{(0.156)} &&  &  \\ 
   \hline
\end{tabular}
\label{tab2}
\end{table}

Figure \ref{filters} reports the estimated conditional mean for the NB-GAS and the BNB-GAS. We can see that the BNB-GAS estimate of $\lambda_t$ is robust to the outliers in  August 2000, March 2008 and May 2015. Instead, on the contrary, the estimate from the negative binomial specification is strongly affected by the outliers. This further highlights the empirical relevance of  BNB autoregressive models in modeling integer-valued time series with extreme values.

\begin{figure}[h!]
\center
\includegraphics[scale=0.65]{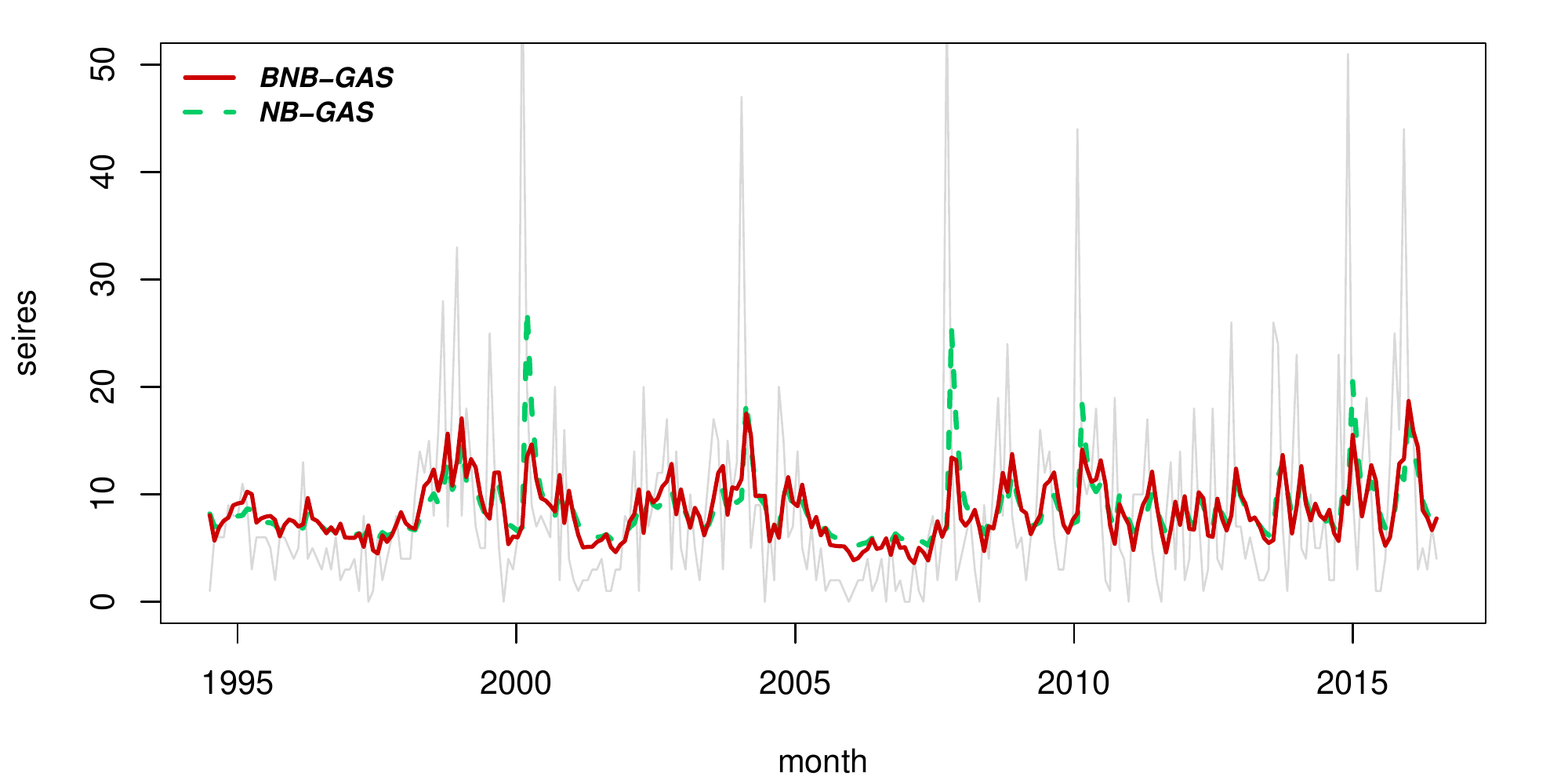}
\vspace{-0.5cm}
\caption{\small{Estimated conditional mean $\lambda_t$ from the BNB-GAS model (in red) and the NB-GAS (in green). The original time series is displayed in gray.} }
\label{filters}
\end{figure}

\section{Conclusion}
\label{section6}
This  article introduces a general framework for modeling  integer-valued time series with outliers. The paper proposes a class of observation-driven models that are based on a mixture of negative binomial distributions, known as the BNB distribution. The stochastic properties of the models and the asymptotic theory of ML estimation are formally discussed.  Two different specifications are considered and studied. An empirical application  illustrates the practical relevance of the approach. Further research may focus on extending the proposed method to more flexible specifications. For instance, relevant developments may include embedding the model with a zero-inflated BNB distribution to handle time series with large numbers of zeros.

\appendix

\section{Appendix}

\subsection{Proofs}

\begin{proof}[Proof of Theorem \ref{th1}]
First we show that (i) holds true. For convenience, we rewrite the equations of the BNB autoregressive process as follows
$$y_t= F_{\lambda_t}^{-1}(u_t), \quad \lambda_{t+1}=g_{{\pmb{\theta}}}(F_{\lambda_t}^{-1}(u_t),\lambda_t),$$
where $\{u_t\}_{t\in\mathbb{Z}}$ is an iid sequence of uniform random variables, $u_t\sim \mathcal{U}(0,1)$, and $F_{\lambda_t}^{-1}(u)=\inf_{z\in\mathbb{R}}\{u\le F_{\lambda_t}(z)\}$, $u\in (0,1)$, where $F_{\lambda_t}(z)$, $z\in\mathbb{R}$, is the cumulative distribution function of a BNB random variable $\mathcal{BNB}(\lambda_t,r,\alpha)$. We prove the result by showing that there is a unique stationary and ergodic causal solution $\{\lambda_t\}_{t\in\mathbb{Z}}$ with $\mathbb{E}(\lambda_t)<\infty$ by an application of Theorem 3.1 in \cite{doukhan2008weakly} (see also Remark 3.1). Then it is immediate to conclude that there is a unique stationary and ergodic causal solution  $\{(y_t,\lambda_t)\}_{t\in\mathbb{Z}}$ and,  given $\alpha>1$, we have  $\mathbb{E}(y_t)<\infty$. In the following, we show that  the conditions (3.1) and (3.2) of Theorem 3.1 in \cite{doukhan2008weakly} are satisfied. Note that (3.3) trivially holds given $\alpha>1$. From the contraction condition in (\ref{contraction}), we obtain that
\begin{align*}
\mathbb{E}\left|g_{\pmb{\theta}}\left(F_{x_1}^{-1}(u_0),x_1\right)-g_{\pmb{\theta}}\left(F_{x_2}^{-1}(u_0),x_2\right)\right| &=\int_{0}^1\left|g_{\pmb{\theta}}\left(F_{x_1}^{-1}(u),x_1\right)-g_{\pmb{\theta}}\left(F_{x_2}^{-1}(u),x_2\right)\right|d u\\
&\le c_1 \left|x_1-x_2\right| +c_2 \int_{0}^1 \left|F_{x_1}^{-1}(u)-F_{x_2}^{-1}(u)\right| du.
\end{align*}
Now, if  $x_1\ge x_2$, by appealing to the stochastic ordering result in Lemma \ref{lemma1},  we obtain that $ F_{x_1}^{-1}(u) \ge F_{x_2}^{-1}(u)$ for any $u\in (0,1)$. Therefore, we have that
\begin{align*}
\mathbb{E}\left|g_{\pmb{\theta}}\left(F_{x_1}^{-1}(u_0),x_1\right)-g_{\pmb{\theta}}\left(F_{x_2}^{-1}(u_0),x_2\right)\right|
&\le c_1 \left|x_1-x_2\right| +c_2 \int_{0}^1\left|F_{x_1}^{-1}(u)-F_{x_2}^{-1}(u)\right| du \\
&= c_1 \left|x_1-x_2\right| +c_2 \left( \int_{0}^1   F_{x_1}^{-1}(u) du -  \int_{0}^1 F_{x_2}^{-1}(u) du\right)\\
&= \left(c_1 +c_2\right)  \left|x_1-x_2\right|,
\end{align*}
where the last equality follows from the fact that $\int_{0}^1   F_{\lambda}^{-1}(u) du =\lambda$. In a similar way, it is straightforward to show that the same result holds also if $x_1\le x_2$ and therefore 
$$\mathbb{E}\left|g_{\pmb{\theta}}\left(F_{x_1}^{-1}(u_0),x_1\right)-g_{\pmb{\theta}}\left(F_{x_2}^{-1}(u_0),x_2\right)\right| \le  \left(c_1 +c_2\right)  \left|x_1-x_2\right|$$ holds for any $x_1, x_2\in \mathbb{R}^+$. As a result, we  obtain that (3.1) and (3.2) are satisfied since $c_1 +c_2<1$ holds by assumption. 

Finally,  we show that (ii) holds by an application of  Theorem 3.1 of \cite{Bougerol1993}. We study  the following stochastic recurrence equation (SRE) for $\lambda_t$
\begin{equation*}
  \lambda_{t+1} =g_{\pmb{\theta}}(y_t,   \lambda_{t} ), \quad t \in \mathbb{Z},
\end{equation*}
where $\{y_t\}_{t\in\mathbb{Z}}$ is the unique stationary and erogodic causal solution of the model. Therefore, $\{g_{\pmb{\theta}}(y_t,  \cdot)\}_{t\in\mathbb{Z}}$ is a stationary and ergodic sequence of functions from $\mathbb{R}^+$ into $\mathbb{R}^+$.
If  the  conditions  C1 and C2 of Theorem 3.1 of \cite{Bougerol1993} are satisfied for  the sequence $\{g_{\pmb{\theta}}(y_t,  \cdot)\}$, then  we obtain that  $\lambda_t$ is $\mathcal{F}_{t-1}$-measurable. Condition C1 is immediately satisfied since, for any $\lambda\in \mathbb{R}^+$, we have
$$\mathbb{E}(\log^+|g_{\pmb{\theta}}(y_t, \lambda) |)\le  |g_{\pmb{\theta}}(0, 0)  | + c_1\mathbb{E} (y_t)+ c_2 \lambda <\infty,$$
where $\mathbb{E} (y_t)<\infty$ follows from (i). As concerns C2, from the contraction condition in (\ref{contraction})  we obtain the following Lipschitz coefficient
$$\frac{|g_{\pmb{\theta}}(y_t, \lambda)-g_{\pmb{\theta}}(y_t, \lambda^*)|}{|\lambda-\lambda^*|}\le c_2 <1.$$
Therefore, we conclude that condition  C2 is satisfied. This concludes the proof of the theorem.
\end{proof}

\begin{proof}[Proof of Theorem \ref{th2}]
In the following, we use the shorthand notation $\mathbb{E}_{t}( \cdot )=\mathbb{E}(\cdot|\mathcal{F}_{t})$  to denote conditional expectations.

The contraction condition in (\ref{contraction}) implies that
\begin{align}\label{ub}
\lambda_t \le c_0 +  c_1 y_{t-1} +c_2 \lambda_{t-1},
\end{align}
where $c_0=g_{\pmb\theta}(0,0)$. Therefore, taking the conditional expectation of $\lambda_t$, we obtain that
\begin{align}
\mathbb{E}_{t-2}(\lambda_t) &\le c_0 +  c_1 \mathbb{E}_{t-2}(y_{t-1}) +c_2 \lambda_{t-1},\nonumber \\
 &= \tilde a_{10} +  \tilde a_{11} \lambda_{t-1}, \label{lambda_ub}
\end{align}
where $\tilde a_{10} = c_0$ and $ \tilde a_{11}= c_1+c_2$.
Similarly, given the inequality in (\ref{ub}), we derive the following upper bound for $\lambda_t^2$
\begin{align*}
\lambda_t^2 \le c_0^2 +  c_1^2 y_{t-1}^2 +c_2^2 \lambda_{t-1}^2 +2 c_0c_1 y_{t-1}+2 c_0c_2 \lambda_{t-1}+2 c_1c_2 y_{t-1} \lambda_{t-1}.
\end{align*}
Now, noticing that the  conditional expectation of $y_t^2$ can be expressed as
\begin{align*}
\mathbb{E}_{t-1}(y_t^2)= \frac{(\alpha+r-1)}{(\alpha-2)}\lambda_t+\frac{(r+1)(\alpha-1)}{r(\alpha-2)}\lambda_t^2,
\end{align*}
we obtain that
\begin{align*}
\mathbb{E}_{t-2}(\lambda_t^2) \le c_0^2 +  c_1^2 \mathbb{E}_{t-2}(y_{t-1}^2) +c_2^2 \lambda_{t-1}^2 +2 c_0c_1 \mathbb{E}_{t-2}(y_{t-1})+2 c_0c_2 \lambda_{t-1}+2 c_1c_2 \mathbb{E}_{t-2}(y_{t-1}) \lambda_{t-1}.
\end{align*}
Calculating the conditional expectations on the right hand side of  the above equation yields 
\begin{align}\label{lambda2_ub}
\mathbb{E}_{t-2}(\lambda_t^2) \le \tilde a_{20} + \tilde a_{21} \lambda_{t-1}+\tilde a_{22} \lambda_{t-1}^2,
\end{align}
where
\begin{align*}
\tilde a_{20}=c_0^2, \quad  \tilde a_{21}=\frac{(\alpha+r-1)}{(\alpha-2)}c_1^2 + 2c_0c_1 +2c_0c_2, \quad \text{and}\quad \tilde a_{22}=\frac{(r+1)(\alpha-1)}{r(\alpha-2)}c_1^2+c_2^2+2c_1c_2.
\end{align*}
Therefore, combining equations (\ref{lambda_ub}) and (\ref{lambda2_ub}), we can write the following bivariate system for $\text{\pmb{x}}_t=(\lambda_t,\lambda^2_t)^\top$
$$\mathbb{E}_{t-2}(\text{\pmb{x}}_t) \le \text{\pmb{d}}_0 +\text{\pmb{D}}_1 \text{\pmb{x}}_{t-1},$$
where 
$$
\text{\pmb{d}}_0=
\begin{bmatrix}
    \tilde a_{10}         \\
    \tilde a_{20}        \\
\end{bmatrix}, \quad
\text{\pmb{D}}_1=
\begin{bmatrix}
    \tilde a_{11}       & 0  \\
    \tilde a_{21}       & \tilde a_{22} \\
\end{bmatrix}.
$$
Reiterating  conditional expectations we obtain that
$$\mathbb{E}_{t-k}(\text{\pmb{x}}_t) \le  \sum_{i=0}^{k-2} \text{\pmb{D}}_1^i \ \text{\pmb{d}}_0+\text{\pmb{D}}_1^{k-1} \text{\pmb{x}}_{t-k+1}.$$
Since  $\text{\pmb{D}}_1$ is a lower triangular matrix, the eigenvalues of $\pmb{D}_1$ are the diagonal entries and by assumption they are smaller than one $\tilde a_{11} \le \tilde a_{22} <1$.
Therefore, taking the limit $k\rightarrow \infty$ we obtain
$$\mathbb{E}(\text{\pmb{x}}_t) \le   (\text{\pmb{I}}-\text{\pmb{D}}_1)^{-1}\text{\pmb{d}}_0,$$
which implies $\mathbb{E}(\lambda_t^2)<\infty$. The final result follows noticing that $\mathbb{E}(\lambda_t^2)<\infty$ entails $\mathbb{E}(y_t^2)<\infty$ since $\alpha>2$. This concludes the proof of the theorem.
\end{proof}

\begin{proof}[Proof of Proposition \ref{invertibility}]

The proof of the invertibility result is obtained by an application of Theorem 3.1 of \cite{Bougerol1993}.   We follow \cite{SM2006} (Proposition 3.12) and  apply Bougerol's theorem in the space of continuous functions $\mathbb{C}(\Theta,\mathbb{R})$ equipped with the uniform  norm $\|\cdot\|_{\Theta}$. Our  SRE   is of the form
\begin{equation*}
\hat \lambda_{t+1}({\pmb{\theta}})=g_{\pmb{\theta}}(y_t, \hat \lambda_{t}({\pmb{\theta}})), \quad t \in \mathbb{N}.
\end{equation*}
First, we note that our SRE satisfies the stationarity and continuity requirements to apply Bougerol's theorem in  $\mathbb{C}(\Theta,\mathbb{R})$. In particular, the  function   $\pmb{\theta}\mapsto \hat \lambda_{t}({\pmb{\theta}})$ is continuous for any $t\in \mathbb{N}$  by Assumption \ref{a1},  and the sequence $\{g_{\pmb{\theta}}(y_t, \lambda)\}$ is stationary and ergodic for any $\lambda \in \mathbb{R}^+$ by Theorem \ref{th1}.  Next, we show that the conditions C1 and C2 in Theorem 3.1 of \cite{Bougerol1993} are satisfied. As concerns C1, we obtain that 
$$\mathbb{E}(\log^+\|g_{\pmb{\theta}}(y_t, \lambda) \|_{\Theta})\le \|g_{\pmb{\theta}}(0, 0) \|_{\Theta}+ c_1\mathbb{E} (y_t)+ c_2 \lambda <\infty $$
since the contraction holds over $\Theta$ by Assumption \ref{a2},  $\Theta$ is compact, and $\mathbb{E}(y_t)<\infty$ by Theorem \ref{th1}. Finally, C2 is immediately satisfied as 
$$\frac{|g_{\pmb{\theta}}(y_t, \lambda)-g_{\pmb{\theta}}(y_t, \lambda^*)|}{|\lambda-\lambda^*|}\le c_2 <1$$
with probability one for any $\pmb{\theta} \in\Theta$ by   Assumption \ref{a2}.

Finally, we conclude the proof by showing that $\mathbb{E}\|\tilde\lambda_t\|_{\Theta}<\infty$. By Assumption \ref{a2}, we obtain that with probability 1
\begin{align*}
\|\tilde\lambda_t \|_{\Theta}&\le c_0 +  c_1\ y_{t-1} + c_2 \ \|\tilde\lambda_{t-1}\|_{\Theta} \\ 
&\le  \sum_{i=0}^{t-1} c_2^i\  (c_0 +c_1\ y_{t-1-i}) +  c_2^t \ \|\tilde\lambda_{0}\|_{\Theta}\\
&\le  \sum_{i=0}^{\infty} c_2^i\  (c_0 +c_1\ y_{t-1-i}) +  c_2^t \ \|\tilde\lambda_{0}\|_{\Theta},
\end{align*}
where $c_0=\|g_{\pmb{\theta}}(0, 0) \|_{\Theta}$. Therefore, since $c_2<1$, for large enough $t$ we obtain that
\begin{align*}
\|\tilde\lambda_t \|_{\Theta}
\le  1+\sum_{i=0}^{\infty} c_2^i\  (c_0 +c_1\ y_{t-1-i}),
\end{align*}
and by stationarity we have that the above inequality holds for any $t$. Therefore, the desired result follows by noticing that $\mathbb{E}(y_t)<\infty$ by Theorem \ref{th1}. 
\end{proof}

\begin{proof}[Proof of Theorem \ref{th3}]
In the following, we  show:
\begin{enumerate}[(a)]
\item Uniform convergence of the likelihood, i.e.~$\|\hat L_T-  L\|_{K}\xrightarrow{a.s.} 0$, as $T\rightarrow \infty$. Here the limit likelihood function $L$ is defined as $L(\pmb{\kappa})=\mathbb{E}(l_t(\pmb{\kappa}))$.
\item Identifiability of $\pmb{\kappa}_0$, i.e.~$L({\pmb{\kappa}}_0)>L({\pmb{\kappa}})$ \ $\forall$ \ ${\pmb{\kappa}}\in K, {\pmb{\kappa}}\neq {\pmb{\kappa}}_0$.
\end{enumerate}
 Then, given  the conditions (a) and (b) and the compactness of $K$,  the consistency  $\hat {\pmb{\kappa}}_T\xrightarrow{a.s.}{\pmb{\kappa}}_0$ follows immediately by  standard arguments that go back to \cite{wald1949}.

\noindent (a)  An application of the triangle inequality yields
\begin{align}
|\hat L_T({\pmb{\kappa}})-L({\pmb{\kappa}})|\le |\hat L_T({\pmb{\kappa}})-L_T({\pmb{\kappa}})|+|L_T({\pmb{\kappa}})-L({\pmb{\kappa}})|,
\label{triangle}
\end{align}
where $L_T({\pmb{\kappa}})=\frac{1}{T}\sum_{t=1}^T l_t(\pmb\kappa)$. 
Therefore, the desired result follows if we can show that both terms on the right hand side of the inequality (\ref{triangle}) are vanishing almost surely and uniformly as  $T$ diverges. 

As concerns the first term in (\ref{triangle}),  Assumption \ref{a3} implies that  $\tilde \lambda_t({\pmb{\theta}})>\bar c$ with probability 1 for any ${\pmb{\theta}} \in \Theta$. An application of the mean value theorem together  with Lemma \ref{lemma100} yields
 \begin{align*}
|\hat l_t({\pmb{\kappa}})-l_t({\pmb{\kappa}})|&\le \sup_{{\pmb{\kappa}} \in K}\sup_{\lambda \in [\bar c,+\infty) }\left|\frac{\partial\log p(y_t|\lambda,r,\alpha)}{\partial \lambda}\right||\hat \lambda_t({\pmb{\theta}})-  \tilde  \lambda_t({\pmb{\theta}})|\\
&\le c_\lambda \ |\hat \lambda_t({\pmb{\theta}})-  \tilde  \lambda_t({\pmb{\theta}})|.
\end{align*}
As a result,  since $\|\hat \lambda_t -  \tilde  \lambda_t \|_\Theta\xrightarrow{e.a.s.}0$ by Proposition \ref{invertibility}, we obtain $\|\hat l_t-l_t\|_K\xrightarrow{a.s.}0$ and therefore
$$\|\hat L_T-L_T\|_K\le \frac{1}{T}\sum_{t=1}^T \|\hat l_t-l_t\|_K\xrightarrow{a.s.}0.$$

As concerns the second term in (\ref{triangle}), given the continuity of the log-likelihood and the compactness of $K$,  we obtain that $\|L_T-L\|_K\xrightarrow{a.s.}0$ follows by an application of the ergodic theorem of  \cite{rao1962}, provided that the log-likelihood has a uniformly bounded moment, i.e.~$\mathbb{E}\|l_t\|_K<\infty$. 
We notice that the log-likelihood can be bounded as follows 
\begin{align*}
|l_t(\pmb{\kappa})|=&\left|\log\frac{\Gamma(y_t+r)}{\Gamma(y_t+1)\Gamma(r)}+\log\frac{B\Big((\alpha-1)\tilde\lambda_t(\pmb\theta)/r+r,(\alpha-1)\tilde\lambda_t(\pmb\theta)/r+y_t\Big)}{B\Big(\alpha ,(\alpha-1)\tilde\lambda_t(\pmb\theta)/r \Big)}\right|\\
\le & |\log B(r,\alpha)|+|\log\Gamma(y_t+r)-\log\Gamma(y_t+1)|\\
& + \Big|\log\Gamma\Big((\alpha-1)\tilde\lambda_t(\pmb\theta)/r+y_t+\alpha+r\Big)-\log\Gamma\Big((\alpha-1)\tilde\lambda_t(\pmb\theta)/r+y_t\Big)\Big|\\
& + \Big|\log\Gamma\Big((\alpha-1)\tilde\lambda_t(\pmb\theta)/r+\alpha\Big)-\log\Gamma\Big((\alpha-1)\tilde\lambda_t(\pmb\theta)/r\Big)\Big|.
\end{align*}
Therefore, $\mathbb{E}\|l_t\|_K<\infty$ follows by an application of Lemma \ref{lemma_gamma} since  $\inf_{\pmb \theta \in \Theta}\tilde \lambda_t(\pmb\theta)\ge \bar c$ a.s.~by Assumption \ref{a3},  $\mathbb{E}(y_t)<\infty$ by Theorem \ref{th1} and $\mathbb{E}\|\tilde \lambda_t\|_{\Theta}< \infty$ by Proposition \ref{invertibility}.
This concludes the proof of (a).

\noindent (b) First, we show that $l_t({\pmb{\kappa}})=l_t({\pmb{\kappa}}_0)$ a.s.~if and only if (iff) ${\pmb{\kappa}}={\pmb{\kappa}}_0$. It is clear that the parameters of the  BNB pmf are well identified. In particular,     we have that
$$  p(y | \tilde \lambda_t({\pmb{\theta}}_1), r_1, \alpha_1) =   p(y | \tilde \lambda_t({\pmb{\theta}}_2), r_2, \alpha_2) \; \forall \;y\in\mathbb{N} \quad \text{iff}\quad (\tilde\lambda_t({\pmb{\theta}}_1), r_1, \alpha_1) = (\tilde\lambda_t({\pmb{\theta}}_2), r_2, \alpha_2).$$
Therefore, the desired result follows if we can show that $\tilde \lambda_t({\pmb{\theta}}_1)=\tilde \lambda_t({\pmb{\theta}}_2)$ iff ${\pmb{\theta}}_1={\pmb{\theta}}_2$ for any $\pmb{\theta}_1,\pmb{\theta}_2\in\Theta$. We prove the result by contradiction.
Assume that ${\pmb{\theta}}_1\neq{\pmb{\theta}}_2$ and $\tilde \lambda_t({\pmb{\theta}}_1)=\tilde \lambda_t({\pmb{\theta}}_2)$ a.s., then, given the stationarity of  $\tilde \lambda_t$, we have that $\tilde \lambda_t({\pmb{\theta}}_1)=\tilde \lambda_t({\pmb{\theta}}_2)$ a.s.~for any $t$. Therefore, we can assume that $\tilde\lambda_{t-1}({\pmb{\theta}}_1)=\tilde\lambda_{t-1}({\pmb{\theta}}_2)=\lambda$ with probability 1 and it must be true that $g_{{\pmb{\theta}}_1}(y_t,\lambda)=g_{{\pmb{\theta}}_2}(y_t,\lambda)$ with probability 1. However, this cannot be true because  Assumption \ref{a4} implies  that $g_{{\pmb{\theta}}_1}(y_t,\lambda)\neq g_{{\pmb{\theta}}_2}(y_t,\lambda)$ with positive probability. We conclude that $\tilde \lambda_t({\pmb{\theta}}_1)=\tilde \lambda_t({\pmb{\theta}}_2)$ iff ${\pmb{\theta}}_1={\pmb{\theta}}_2$  and therefore $l_t({\pmb{\kappa}})=l_t({\pmb{\kappa}}_0)$ a.s.~iff ${\pmb{\kappa}}={\pmb{\kappa}}_0$.

Finally, we  show that $l_t({\pmb{\kappa}})=l_t({\pmb{\kappa}}_0)$ a.s.~iff ${\pmb{\kappa}}={\pmb{\kappa}}_0$ entails the identifiability condition (b).
 It is well known that $\log(x)\le x-1$ for any $x\in \mathbb{R}^+$ with the equality  only in the case $x=1$. This implies that almost surely
\begin{eqnarray}\label{ineq}
l_t({\pmb{\kappa}})-l_t({\pmb{\kappa}}_0)\le\frac{p(y_t|\tilde \lambda_t({\pmb{\theta}}),r,\alpha)}{p(y_t|  \lambda_t,r_0,\alpha_0)}-1.
\end{eqnarray}
Moreover, we have that  the inequality in (\ref{ineq}) holds as a strict inequality with positive probability because  $p(y_t|\tilde \lambda_t({\pmb{\theta}}),r,\alpha)\neq p(y_t|\lambda_t,r_0,\alpha_0)$ with positive probability for any ${\pmb{\kappa}}\in K$, ${\pmb{\kappa}}\neq{\pmb{\kappa}}_0$. As a result
\begin{eqnarray*}
\mathbb{E}\left[\mathbb{E}\left[l_t({\pmb{\kappa}})-l_t({\pmb{\kappa}}_0)|\mathcal{F}_{t-1}\right]\right] < \mathbb{E}\left[\mathbb{E}\left[\frac{p(y_t|\tilde \lambda_t({\pmb{\theta}}),r,\alpha)}{p(y_t|  \lambda_t,r_0,\alpha_0)}\Big|\mathcal{F}_{t-1}\right]\right]-1=0, \quad    \forall \ {\pmb{\kappa}} \neq {\pmb{\kappa}}_{0},
\end{eqnarray*}
where the right hand side of the inequality is equal to zero as $p(y_t|  \lambda_t ,r_0,\alpha_0)$ is the true conditional pmf.
Finally, the desired result $L({\pmb{\kappa}}_0)>L({\pmb{\kappa}})$ follows  by the law of total expectation
$$L({\pmb{\kappa}})-L({\pmb{\kappa}}_0)=\mathbb{E}\left[\mathbb{E}\left[l_t({\pmb{\kappa}})-l_t({\pmb{\kappa}}_0)|\mathcal{F}_{t-1}\right]\right]<0 \quad   \forall \ {\pmb{\kappa}} \neq {\pmb{\kappa}}_{0}.$$
This concludes the proof of (b).
\end{proof}

\begin{proof}[Proof of Theorem \ref{th4}]
For the asymptotic normality proof, we follow a similar argument as the in proof in Section 7 of \cite{SM2006}. First, we derive the asymptotic distribution of the ML estimator $\tilde{\pmb{\kappa}}_T$ based on the limit likelihood $L_T$, which is defined as
$$\tilde{\pmb\kappa}_T=\argsup_{\pmb{\kappa} \in K}L_{T}(\pmb{\kappa}).$$
Then, we show that $\hat{\pmb\kappa}_T$ and $\tilde{\pmb\kappa}_T$ have the same asymptotic distribution.

From the uniform convergence in Lemmas \ref{lemma_n1.1} and  \ref{lemma_n1.2}, we have that $\tilde \lambda_t$ is twice continuosly differentiable in $\Theta$ with first and second derivatives given by $\tilde \lambda_t' $ and $\tilde \lambda_t'' $, respectively. This  immediately implies that the limit likelihood $L_T$ is   twice continuously differentiable in the compact set $K$. Therefore, a Taylor expansion around ${\pmb\kappa}_0$ yields
$$ L_T'(\tilde {\pmb{\kappa}}_T)=  L_T'({\pmb{\kappa}}_0)+  L_T''(\pmb{\eta}_T)(\tilde {\pmb{\kappa}}_T-{\pmb{\kappa}}_0),$$
where $\pmb{\eta}_T$ is a point between $\tilde {\pmb{\kappa}}_T$ and ${\pmb{\kappa}}_0$. By definition, $\tilde {\pmb{\kappa}}_T$ is the maximizer of $L_T({\pmb{\kappa}})$. Therefore, we have that $ L_T'(\tilde {\pmb{\kappa}}_T)=0$ for large enough $T$ since $\tilde {\pmb{\kappa}}_T \xrightarrow{a.s.}{\pmb{\kappa}}_0$ and ${\pmb{\kappa}}_0\in \text{int}(K)$. As a result, the following equation holds true
$$\sqrt{T}  L_T'({\pmb{\kappa}}_0)=-  L_T''(\pmb{\eta}_T) \sqrt{T}(\hat {\pmb{\kappa}}_T-{\pmb{\kappa}}_0).$$

By Lemma \ref{lemma_n2} together with an application of the ergodic theorem of  \cite{rao1962}, we obtain that $-  L_T''(\pmb{\eta}_T) \xrightarrow{a.s.}\pmb{\text{F}}_0$, where $\pmb{\text{F}}_0=-\mathbb{E}[l_t''(\pmb\kappa_0)]$. Furthermore, Lemma \ref{lemma_n3} ensures that  $\pmb{\text{F}}_0$ is positive definite and Lemma \ref{lemma_n0} shows that  $\sqrt{T}  L_T'({\pmb{\theta}}_0)\xrightarrow{d}N(\pmb{0}, \pmb{\text{F}}_0)$. Therefore, we get that 
$$\sqrt{T}(\tilde {\pmb{\kappa}}_T-{\pmb{\kappa}}_0) =  \pmb{\text{F}}_0^{-1} \sqrt{T}  L_T'({\pmb{\kappa}}_0)+o_p(1),$$
which implies $\sqrt{T}(\tilde {\pmb{\kappa}}_T-\pmb\kappa_0)\xrightarrow{d}N(\pmb{0}, \pmb{\text{F}}_0^{-1})$ as $T\rightarrow \infty$.

We conclude the proof  by showing that $\sqrt{T}(\tilde {\pmb{\kappa}}_T-\pmb\kappa_0)\xrightarrow{d}N(\pmb{0}, \pmb{\text{F}}_0^{-1})$ entails $\sqrt{T}(\hat {\pmb{\kappa}}_T-\pmb\kappa_0)\xrightarrow{d}N(\pmb{0}, \pmb{\text{F}}_0^{-1})$. A Taylor expansion yields
$$L_T'(\hat {\pmb{\kappa}}_T)=  L_T'(\tilde {\pmb{\kappa}}_T) +   L_T''(\tilde{\pmb{\eta}}_T)(\hat {\pmb{\kappa}}_T-\tilde {\pmb{\kappa}}_T),$$
where $\tilde{\pmb{\eta}}_T$ is a point between $\tilde {\pmb{\kappa}}_T$ and $\hat {\pmb{\kappa}}_T$. Furthermore, we note that   $\hat L_T'(\hat {\pmb{\kappa}}_T)=0$ and $L_T'(\tilde {\pmb{\kappa}}_T)=0$ for large enough $T$ since the estimators are strongly consistent and $\pmb{\kappa}_0\in \text{int}(\pmb{K})$. Therefore, we have that
$$\sqrt{T}\big(L_T'(\hat {\pmb{\kappa}}_T)- \hat L_T'(\hat {\pmb{\kappa}}_T)\big)=    L_T''(\tilde{\pmb{\eta}}_T)\sqrt{T}(\hat {\pmb{\kappa}}_T-\tilde {\pmb{\kappa}}_T).$$
The left hand side of the above equation goes to zero almost surely as $T\rightarrow \infty$ by an application of  Lemma \ref{lemma_n1,5}. Furthermore, Lemma \ref{lemma_n2} ensures that $L_T''(\tilde{\pmb{\eta}}_T)\xrightarrow{a.s.}-\pmb{\text{F}}_0$. Therefore, we obtain that $\sqrt{T}(\hat {\pmb{\kappa}}_T-\tilde {\pmb{\kappa}}_T)\xrightarrow{a.s.}0$. This concludes the proof of the theorem since  $$\sqrt{T}(\hat {\pmb{\kappa}}_T-  {\pmb{\kappa}}_0)= \sqrt{T}(\tilde {\pmb{\kappa}}_T-  {\pmb{\kappa}}_0)+\sqrt{T}(\hat {\pmb{\kappa}}_T-\tilde {\pmb{\kappa}}_T)=\sqrt{T}(\tilde {\pmb{\kappa}}_T-  {\pmb{\kappa}}_0)+o_p(1).$$
\end{proof}

\begin{proof}[Proof of Theorem \ref{p1}] The proof follows immediately by an application of Theorems \ref{th1} and  \ref{th2}. The contraction condition of  Theorems \ref{th1}  is immediately satisfied  if $\tau +\phi<1$. Instead,  
$$\frac{(r+1)(\alpha-1)}{r(\alpha-2)} \tau^2 +\phi^2+2\tau \phi<1,$$ 
ensures that the contraction condition of Theorem  \ref{th2} holds.
\end{proof}

\begin{proof}[Proof of Theorem \ref{th:ML}] First, we obtain the consistency result by showing that Assumptions \ref{a1}-\ref{a4} are  satisfied.   Assumptions \ref{a1} and \ref{a2} are trivially satisfied since the updating function of the  BNB-INGARCH is continuous and the parameter  set $K$ is  such that $\phi+\tau<1$. Assumption \ref{a3} is satisfied  since $\omega>0$ for any $\pmb \kappa \in K$ and $K$ is compact. Finally, Assumption \ref{a4} is trivially satisfied by the functional form of the updating function of the model. Therefore,   the ML estimator is strongly consistent by an application of Theorem \ref{th1}.

Next, we obtain the asymptotic normality by showing that Assumptions \ref{a5.0}-\ref{a6} are satisfied. Assumption \ref{a5.0} holds since $\pmb {\kappa}_0$ satisfies the contraction condition for weak stationarity. Assumption \ref{a5} holds since $1$, $y_t$ and $\lambda_t$ are linearly independent random variables. Finally, Assumption \ref{a6} is satisfied since $g_{\pmb \theta}(y,\lambda)$ is  a  linear  function of  $(y,\lambda)$  and  therefore any derivative is bounded by a linear combination of $(y,\lambda)$. As a result,  the ML estimator is asymptotically normal by an application of Theorem \ref{th2}. This concludes the proof of the theorem.
\end{proof}

\begin{proof}[Proof of Theorem \ref{SE_BNBINGAS}]  We denote the updating function of $\lambda_t$ in (\ref{BNB-GAS}) as $g(y,\lambda)= \exp\big(\omega+\phi \log \lambda +\tau s(y,\lambda) \big)$, where $s(y,\lambda)= \partial \log p(y|\lambda,r, \alpha)/ \partial \log \lambda$.
First, given the expression of the score innovation in (\ref{score}) and  inequality   (\ref{in2}), we note that  $-(r+\alpha+1) \le s(y,\lambda)\le (\alpha+1)$ for any $y\in\mathbb{N}$ and $\lambda\in\mathbb{R}^+$.  Therefore, given that $0\le \phi <1$, we immediately obtain that $\lambda_t$ is bounded from below  with probability one by the constant $\exp\big((\omega-\tau(\alpha+r+1))/(1-\phi)\big)$. Next, we apply the mean value theorem to $g(\cdot,\cdot)$ and obtain that
\begin{align*}
|g(y,\lambda)-g(y^*,\lambda^*)|\le \sup_{y,\lambda}\left|\frac{\partial g(y,\lambda)}{\partial y}\right| |y-y^*|+\sup_{y,\lambda}\left|\frac{\partial g(y,\lambda)}{\partial \lambda}\right| |\lambda-\lambda^*|.
\end{align*}
Therefore, we are only left with showing that $\sup_{\lambda,y}\left| {\partial g(y,\lambda)}/{\partial y}\right|+\sup_{\lambda,y}\left| {\partial g(y,\lambda)}/{\partial \lambda}\right|<1$ since the desired result then follows by an application of Theorem \ref{th1}. Below, we show that this is the case. In particular, we obtain that
\begin{align*}
\left|\frac{\partial g(y,\lambda)}{\partial y}\right|  =& g(y,\lambda)\tau\gamma\lambda\Big(\psi_1(\gamma\lambda+y)-\psi_1(\gamma\lambda+y+r+\alpha)\Big)\\
\le& \tau \exp\big(\omega+\tau s(y,\lambda)\big)\gamma \lambda^{1+\phi}\frac{(\alpha+r+1)}{(\gamma \lambda+y)^2}\\
\le& \tau \exp(\omega+\tau\alpha+\tau)\frac{(\alpha+r+1)}{\gamma \lambda^{1-\phi}}\\
\le& \tau \exp(\omega+\tau\alpha+\tau)\frac{(\alpha+r+1)}{\gamma \exp\big(\omega-\tau(\alpha+r+1)\big)}\\
=&  {\tau \gamma^{-1} (\alpha+r+1) \exp\big(\tau (2\alpha+r+2)\big)},
\end{align*}
where the first inequality follows by the inequality in (\ref{in3}), and the second and third inequalities follow by taking the supremum over $y$ and $\lambda$.
In a similar way, we obtain
\begin{align*}
\left|\frac{\partial g(y,\lambda)}{\partial \lambda}\right|  =& g(y,\lambda) \left(\frac{\phi}{\lambda} +\tau \left|\frac{\partial s(y,\lambda)}{\partial \lambda}\right| \right)\\
\le & \exp\big(\omega+\tau s(y,\lambda)\big)\lambda^{\phi}\left(\frac{\phi}{\lambda}+\frac{\tau |s(y,\lambda)|}{\lambda}+\tau\gamma^2\lambda\big(\psi_1(\gamma\lambda)-\psi_1(\gamma\lambda+\alpha)\big)\right)\\
& + \exp\big(\omega+\tau s(y,\lambda)\big)\lambda^{\phi}\Big(\tau \gamma^2\lambda\big(\psi_1(\gamma\lambda+y)-\psi_1(\gamma\lambda+y+r+\alpha)\big)\Big)\\
\le&  \exp(\omega+\tau \alpha+\tau)\left(\frac{\phi}{\lambda^{1-\phi}}+\frac{\tau (\alpha+r+1)}{\lambda^{1-\phi}}+\frac{\tau (\alpha+1)}{ \lambda^{1-\phi}}+\frac{\tau (\alpha+r+1) \gamma^2 \lambda^{1+\phi}}{(\gamma\lambda +y)^2}\right)\\
\le&  \exp(\omega+\tau \alpha+\tau) \frac{\phi+\tau(3\alpha+2r+3)}{\lambda^{1-\phi}}\\
\le& \big({\tau (3\alpha+2r+3)+\phi \big) \exp\big(\tau (2\alpha+r+2)\big)},
\end{align*}
where the first inequality follows by the triangle inequality, the second follows by  the inequality in (\ref{in3}), and the third and fourth  by taking the supremum over $y$ and $\lambda$.
This concludes the proof of the theorem.
\end{proof}

\subsection{Lemmas}

\begin{lemma}\label{lemma_n1.1}
Let the assumptions of Theorem \ref{th4} hold, then 
$$\|\hat \lambda_t' -\tilde \lambda_t'\|_{\Theta}\xrightarrow{e.a.s.}0\quad  \text{as} \quad t\rightarrow \infty,$$ 
where $\tilde \lambda_t'$  is the  stationary and ergodic derivative process of $\tilde \lambda_t$.
Furthermore,    $\tilde \lambda_t'$ has a uniformly bounded second moment, i.e.~$\mathbb{E}\|\tilde\lambda_t'\|_\Theta^2 <\infty$.
\end{lemma}

\begin{proof}
We prove the convergence result $\|\hat \lambda_t' -\tilde \lambda_t'\|_{\Theta}\xrightarrow{e.a.s.}0$ by showing that the conditions S.1-S.3 of  Theorem 2.10 in \cite{SM2006}  are satisfied. In particular, the expression of the first derivative process is 
$$\tilde {\lambda}_{t+1}'({\pmb{\theta}})= g_{\pmb{\theta}t}^{{\pmb{\theta}}} (\pmb\theta)  + \  g_{\pmb{\theta}t}^\lambda(\pmb\theta) \ \tilde {\lambda}_t'({\pmb{\theta}}),$$
where $g_{\pmb{\theta}t}^{{\pmb{\theta}}}(\pmb\theta)$ and $g_{\pmb{\theta}t}^\lambda(\pmb\theta)$ are shorthand notation  for  $g_{\pmb{\theta}}^{{\pmb{\theta}}}\big(y_t, \tilde {\lambda}_t(\pmb\theta)\big)$ and  $g_{\pmb{\theta}}^{{\lambda}}\big(y_t, \tilde {\lambda}_t(\pmb\theta)\big)$, respectively.
We note that conditions S.1 and S.2 are immediately satisfied because $\|g_{\pmb{\theta}t}^\lambda\|_\Theta \le c_1<1$  a.s.~by  Assumption \ref{a2} and 
$$\mathbb{E}\|g_{\pmb{\theta}t}^{{\pmb{\theta}}}\|_\Theta \le \|g_{\pmb{\theta}}^{{\pmb{\theta}}}(0,0)\|_\Theta+ b_2 \mathbb{E}\|\tilde\lambda_t\|_{\Theta}+ b_1\mathbb{E}(y_t)< \infty,$$
by Assumption \ref{a6}. Next we show that S.3 is satisfied,  which is the equivalent of showing $\|\hat g_{\pmb{\theta}t}^{\pmb\theta}-g_{\pmb{\theta}t}^{\pmb\theta}\|_{\Theta}\xrightarrow{e.a.s.}0$ and $\|\hat g_{\pmb{\theta}t}^{\lambda}-g_{\pmb{\theta}t}^{\lambda}\|_{\Theta}\xrightarrow{e.a.s.}0$, where $\hat g_{\pmb{\theta}t}^{{\pmb{\theta}}}(\pmb\theta)$ and $\hat g_{\pmb{\theta}t}^\lambda(\pmb\theta)$ denote $g_{\pmb{\theta}}^{{\pmb{\theta}}}\big(y_t, \hat {\lambda}_t(\pmb\theta)\big)$ and  $g_{\pmb{\theta}}^{{\lambda}}\big(y_t, \hat {\lambda}_t(\pmb\theta)\big)$, respectively. By Assumption \ref{a6} together with Proposition \ref{invertibility}, we obtain 
$$\|\hat g_{\pmb{\theta}t}^{\pmb\theta}-g_{\pmb{\theta}t}^{\pmb\theta}\|_{\Theta}\le b_2\|\hat \lambda_t-\tilde \lambda_t\|_\Theta\xrightarrow{e.a.s.}0,$$
and 
$$\|\hat g_{\pmb{\theta}t}^{\lambda}-g_{\pmb{\theta}t}^{\lambda}\|_{\Theta}\le b_2\|\hat \lambda_t-\tilde \lambda_t\|_\Theta\xrightarrow{e.a.s.}0. $$
Therefore, we conclude that S.3 is satisfied.

Finally, we show that  $\mathbb{E}\|\tilde{\lambda}_{t}'\|_{\Theta}^2<\infty$. Given the  expression of $\tilde \lambda_t'$ and noticing that $\|g_{\pmb{\theta}t}^\lambda\|_\Theta\le c_2<1$  by Assumption \ref{a2}, we obtain   
\begin{align*}
\|\tilde{\lambda}_{t}'\|_{\Theta} &\le  \|g_{\pmb{\theta}t-1}^{{\pmb{\theta}}}\|_{\Theta}  + c_2  \|\tilde {\lambda}_{t-1}'\|_{\Theta}\\
&\le  \sum_{i=0}^{t-1}c_2^i \|g_{\pmb{\theta}t-i-1}^{{\pmb{\theta}}}\|_{\Theta} +  c_2^t  \|\tilde{\lambda}_{0}'\|_\Theta\\
&\le  \sum_{i=0}^{\infty}c_2^i \|g_{\pmb{\theta}t-i-1}^{{\pmb{\theta}}}\|_{\Theta} +  c_2^t  \|\tilde{\lambda}_{0}'\|_\Theta.
\end{align*}
Since $c_2<1$, for large enough $t$ we have that 
$$ \|\tilde{\lambda}_{t}'\|_{\Theta} \le 1+ \sum_{i=0}^{\infty}c_2^i \|g_{\pmb{\theta}t-i-1}^{{\pmb{\theta}}}\|_{\Theta}.$$
Therefore, by stationarity we conclude that the above inequality holds for any $t$. The desired result  $\mathbb{E}\|\tilde{\lambda}_{t}'\|_{\Theta}^2<\infty$ follows since $\mathbb{E}\|g_{\pmb{\theta}t}^{{\pmb{\theta}}}\|_\Theta^2<\infty$ holds by Assumption \ref{a6} together with $\mathbb{E}\|\tilde \lambda_t\|_{\Theta}^2<\infty$ and $\mathbb{E}(y_t^2)<\infty$, which hold true by Assumption \ref{a5}. This concludes the proof of the Lemma.
\end{proof}

\begin{lemma}\label{lemma_n1.2}
Let the assumptions of Theorem \ref{th4} hold, then 
$$\|\hat \lambda_t'' -\tilde \lambda_t''\|_{\Theta}\xrightarrow{e.a.s.}0\quad  \text{as} \quad t\rightarrow \infty,$$ 
where $\tilde \lambda_t''$ is the  stationary and ergodic second derivative processes of $\tilde \lambda_t$.
Furthermore,    $\tilde \lambda_t''$ has a uniformly bounded first moment, i.e.~$\mathbb{E}\|\tilde\lambda_t''\|_\Theta<\infty$.
\end{lemma}

\begin{proof}
As in the proof of Lemma \ref{lemma_n1.1}, the convergence result $\|\hat \lambda_t'' -\tilde \lambda_t''\|_{\Theta}\xrightarrow{e.a.s.}0$ is obtained by checking the conditions S.1-S.3  of Theorem 2.10 in \cite{SM2006}. In particular,  the expression of the second derivative process  is
\begin{align*}
{\tilde\lambda}_{t+1}''({\pmb{\theta}})=\pmb{A}_t(\pmb{\theta})+g_{\pmb{\theta}t}^{\lambda}(\pmb{\theta}) \ {\tilde\lambda}_{t}''({\pmb{\theta}}),
\end{align*}
with $\pmb{A}_t(\pmb{\theta})$ given by
$$\pmb{A}_t(\pmb{\theta}) = g_{\pmb{\theta}t}^{\pmb{\theta}\pmb{\theta}}(\pmb{\theta})+   g_{\pmb{\theta}t}^{\pmb{\theta}\lambda}(\pmb{\theta}) \ {\tilde\lambda}_t'({\pmb{\theta}})^\top+ {\tilde\lambda}_t'({\pmb{\theta}}) \ g_{\pmb{\theta}t}^{\pmb{\theta}\lambda}(\pmb{\theta})^\top+g_{\pmb{\theta}t}^{\lambda\lambda}(\pmb{\theta})\ {\tilde\lambda}_{t}'({\pmb{\theta}}){\tilde\lambda}_{t}'({\pmb{\theta}})^\top,$$
where $g_{\pmb{\theta}t}^{\pmb{\theta}\pmb{\theta}}(\pmb{\theta})$, $g_{\pmb{\theta}t}^{\pmb{\theta}\lambda}(\pmb{\theta})$ and $g_{\pmb{\theta}t}^{\lambda \lambda}(\pmb{\theta})$ denote $g_{\pmb{\theta}t}^{\pmb{\theta}\pmb{\theta}}\big(y_t,\tilde \lambda_t(\pmb{\theta})\big)$, $g_{\pmb{\theta}t}^{\pmb{\theta}\lambda}\big(y_t,\tilde \lambda_t(\pmb{\theta})\big)$ and $g_{\pmb{\theta}t}^{\lambda\lambda}\big(y_t,\tilde \lambda_t(\pmb{\theta})\big)$, respectively. First, we obtain that S.1 is satisfied by showing that $\mathbb{E}\|\pmb{A}_{t}\|_\Theta<\infty$. In particular, $\|g_{\pmb{\theta}t}^{\lambda\lambda}\|\le b_2$ a.s.~by Assumption \ref{a6}, therefore we have that
\begin{align*}
\mathbb{E}\|\pmb{A}_{t}\|_\Theta & \le \mathbb{E}\|g_{\pmb{\theta}t}^{\pmb{\theta}\pmb{\theta}}\|_\Theta+\mathbb{E} \| g_{\pmb{\theta}t}^{\pmb{\theta}\lambda} \ \tilde{\lambda}_t'^\top\|_\Theta+\mathbb{E} \| \tilde{\lambda}_t' \ g_{\pmb{\theta}t}^{\pmb{\theta}\lambda \top}\|_\Theta+\mathbb{E}\|g_{\pmb{\theta}t}^{\lambda\lambda}\ \tilde {\lambda}_{t}' \tilde{\lambda}_{t}'^\top\|_\Theta\\
&\le  \mathbb{E}\|g_{\pmb{\theta}t}^{\pmb{\theta}\pmb{\theta}}\|_\Theta+2 b_2 \mathbb{E} \|\tilde{\lambda}_t'\|_\Theta+b_2\mathbb{E}\|\tilde{\lambda}_{t}'\|_\Theta^2.
\end{align*}
The desired result follows since  $\mathbb{E}\|\tilde{\lambda}_{t}'\|_\Theta^2<\infty$ by Lemma \ref{lemma_n1.1}, and $\mathbb{E}\|g_{\pmb{\theta}t}^{\pmb{\theta}\pmb{\theta}}\|_\Theta<\infty$ holds true because, by Assumption \ref{a6}, $g_{\pmb{\theta}t}^{\pmb{\theta}\pmb{\theta}}$ is bounded by a linear combination of $y_t$ and $\tilde \lambda_t$, which have bounded moments. 

Second, we  obtain that S.2 is satisfied since  $\|g_{\pmb{\theta}t}^{\lambda}\|_{\Theta}\le c_2<1$ a.s.~by the contraction condition in Assumption \ref{a2}.

Third, we have that the condition S.3 is satisfied if 
\begin{align*}
\|\hat g_{\pmb{\theta}t}^{\pmb{\theta}\pmb{\theta}}-g_{\pmb{\theta}t}^{\pmb{\theta}\pmb{\theta}}\|_{\Theta}\xrightarrow{e.a.s.}0, \quad
\| \hat g_{\pmb{\theta}t}^{\pmb{\theta}\lambda}\ \hat {\lambda}_t'^\top- g_{\pmb{\theta}t}^{\pmb{\theta}\lambda} \ \tilde {\lambda}_t'^\top\|_{\Theta}\xrightarrow{e.a.s.}0,\quad \text{and}\\
\|\hat g_{\pmb{\theta}t}^{\lambda\lambda}\ \hat{\lambda}_{t}' \hat{\lambda}_{t}'^\top-g_{\pmb{\theta}t}^{\lambda\lambda}\ \tilde{\lambda}_{t}'\tilde{\lambda}_{t}'^\top\|_{\Theta}\xrightarrow{e.a.s.}0, \quad \text{as}\quad t\rightarrow \infty,
\end{align*}
where $\hat g_{\pmb{\theta}t}^{\pmb{\theta}\pmb{\theta}}(\pmb{\theta})$, $\hat g_{\pmb{\theta}t}^{\pmb{\theta}\lambda}(\pmb{\theta})$ and $\hat g_{\pmb{\theta}t}^{\lambda \lambda}(\pmb{\theta})$ denote $g_{\pmb{\theta}t}^{\pmb{\theta}\pmb{\theta}}\big(y_t,\hat \lambda_t(\pmb{\theta})\big)$, $g_{\pmb{\theta}t}^{\lambda\pmb{\theta}}\big(y_t,\hat \lambda_t(\pmb{\theta})\big)$ and $g_{\pmb{\theta}t}^{\lambda\lambda}\big(y_t,\hat \lambda_t(\pmb{\theta})\big)$, respectively. Note that $\|\hat g_{\pmb{\theta}t}^{\lambda}-g_{\pmb{\theta}t}^{\lambda}\|_{\Theta}\xrightarrow{e.a.s.}0$ holds true as shown in the proof of Lemma \ref{lemma_n1.1}.  By Assumption \ref{a6}, we obtain
$$\|\hat g_{\pmb{\theta}t}^{\pmb{\theta}\pmb{\theta}}-g_{\pmb{\theta}t}^{\pmb{\theta}\pmb{\theta}}\|_{\Theta}\le b_2 \|\hat \lambda_t -\tilde\lambda_t\|_{\Theta}\xrightarrow{e.a.s.}0.$$
As concerns $\| \hat g_{\pmb{\theta}t}^{\pmb{\theta}\lambda}\ \hat {\lambda}_t'^\top- g_{\pmb{\theta}t}^{\pmb{\theta}\lambda} \ \tilde {\lambda}_t'^\top\|_{\Theta}\xrightarrow{e.a.s.}0$, by Assumption \ref{a6},  we obtain that for large enough $t$
\begin{align*}
\| \hat g_{\pmb{\theta}t}^{\pmb{\theta}\lambda}\ \hat {\lambda}_t'^\top- g_{\pmb{\theta}t}^{\pmb{\theta}\lambda} \ \tilde {\lambda}_t'^\top\|_{\Theta} &\le  \|  g_{\pmb{\theta}t}^{\pmb{\theta}\lambda}(\hat {\lambda}_t'^\top-  \tilde {\lambda}_t'^\top)\|_{\Theta}+\| (\hat g_{\pmb{\theta}t}^{\pmb{\theta}\lambda}- g_{\pmb{\theta}t}^{\pmb{\theta}\lambda})\ \hat {\lambda}_t'^\top\|_{\Theta}\nonumber\\
&\le  \|  g_{\pmb{\theta}t}^{\pmb{\theta}\lambda}\|_{\Theta} \|\hat {\lambda}_t'-  \tilde {\lambda}_t'\|_{\Theta}+(\|\tilde \lambda_t'\|_{\Theta}+1)\| \hat g_{\pmb{\theta}t}^{\pmb{\theta}\lambda}- g_{\pmb{\theta}t}^{\pmb{\theta}\lambda}\|_{\Theta}\nonumber \\
&\le b_2 \|\hat {\lambda}_t'-  \tilde {\lambda}_t'\|_{\Theta}+b_2  (\|\tilde \lambda_t'\|_{\Theta}+1)\| \hat\lambda_t- \tilde\lambda_t\|_{\Theta}. 
\end{align*}
The  result follows  by an application of Lemma 2.1 of \cite{SM2006} since $\| \hat\lambda_t- \tilde\lambda_t\|_{\Theta}\xrightarrow{e.a.s.}0$ by Proposition \ref{invertibility}, $\{\|\tilde \lambda_t'\|_{\Theta}\}_{t\in\mathbb{Z}}$ is stationary and ergodic, and $\mathbb{E}\|\tilde \lambda_t'\|_{\Theta}<\infty$ by Lemma \ref{lemma_n1.1}. As concerns $\|\hat g_{\pmb{\theta}t}^{\lambda\lambda}\ \hat{\lambda}_{t}' \hat{\lambda}_{t}'^\top-g_{\pmb{\theta}t}^{\lambda\lambda}\ \tilde{\lambda}_{t}'\tilde{\lambda}_{t}'^\top\|_{\Theta}\xrightarrow{e.a.s.}0$, by Assumption \ref{a6},  we obtain that for large enough $t$
\begin{align*}
\|\hat g_{\pmb{\theta}t}^{\lambda\lambda}\ \hat{\lambda}_{t}' \hat{\lambda}_{t}'^\top-g_{\pmb{\theta}t}^{\lambda\lambda}\ \tilde{\lambda}_{t}'\tilde{\lambda}_{t}'^\top\|_{\Theta}%
&\le  \|g_{\pmb{\theta}t}^{\lambda\lambda}\ (\hat{\lambda}_{t}' \hat{\lambda}_{t}'^\top- \tilde{\lambda}_{t}'\tilde{\lambda}_{t}'^\top)\|_{\Theta}+\|(\hat g_{\pmb{\theta}t}^{\lambda\lambda}-g_{\pmb{\theta}t}^{\lambda\lambda}) \hat{\lambda}_{t}'\hat{\lambda}_{t}'^\top\|_{\Theta}\nonumber\\%
&\le  \|g_{\pmb{\theta}t}^{\lambda\lambda}\|_{\Theta} \ \|\hat{\lambda}_{t}' \hat{\lambda}_{t}'^\top- \tilde{\lambda}_{t}'\tilde{\lambda}_{t}'^\top\|_{\Theta}+(\|\tilde \lambda_t'\|_{\Theta}^2+1)\| \hat g_{\pmb{\theta}t}^{\lambda\lambda}- g_{\pmb{\theta}t}^{\lambda\lambda}\|_{\Theta}\nonumber\\%
&\le  2b_2(\|\tilde \lambda_t'\|_\Theta+1)  \|\hat{\lambda}_{t}' - \tilde{\lambda}_{t}'\|_{\Theta}+b_2 (\|\tilde \lambda_t'\|_{\Theta}^2+1)\| \hat \lambda_t- \tilde \lambda_t\|_{\Theta}.%
\end{align*}
The result follows by Lemma 2.1 of \cite{SM2006} since $\| \hat\lambda_t- \tilde\lambda_t\|_{\Theta}\xrightarrow{e.a.s.}0$ by Proposition \ref{invertibility}, $\{\|\tilde \lambda_t'\|_{\Theta}^2\}_{t\in\mathbb{Z}}$ is stationary and ergodic, and $\mathbb{E}\|\tilde \lambda_t'\|_{\Theta}^2<\infty$ by Lemma \ref{lemma_n1.1}. Therefore, we conclude that the conditions S.1-S.3  of Theorem 2.10 in \cite{SM2006} are satisfied.

Finally, we show that $\mathbb{E}\|\tilde\lambda_t''\|_\Theta<\infty$. We note that $\|g_{\pmb{\theta}t}^\lambda\|_\Theta\le c_2 < 1$ a.s.~by Assumption \ref{a2}. Therefore,  we obtain that 
\begin{align*}
\|\tilde{\lambda}_{t}''\|_{\Theta} &\le  \|\pmb{A}_{t-1}\|_{\Theta}  + c_2  \|\tilde {\lambda}_{t-1}''\|_{\Theta}\\
&\le  \sum_{i=0}^{t-1}c_2^i \|\pmb{A}_{t-i-1}\|_{\Theta} +  c_2^t  \|\tilde{\lambda}_{0}''\|_\Theta\\
&\le  \sum_{i=0}^{\infty}c_2^i \|\pmb{A}_{t-i-1}\|_{\Theta} +  c_2^t  \|\tilde{\lambda}_{0}''\|_\Theta.
\end{align*}
Since $c_2<1$, for large enough $t$ we have that 
$$ \|\tilde{\lambda}_{t}''\|_{\Theta} \le 1+ \sum_{i=0}^{\infty}c_2^i \|\pmb{A}_{t-i-1}\|_{\Theta}.$$
Therefore, by stationarity we conclude that the above inequality holds for any $t$.
The final result $\mathbb{E}\|\tilde{\lambda}_t''\|_\Theta<\infty$ follows since  $\mathbb{E}\|\pmb{A}_{t}\|_\Theta<\infty$ holds true as shown in the proof of S.1 given above. 
\end{proof}

\begin{lemma}\label{lemma_n2}
Let the assumptions of Theorem \ref{th4} hold, then the second derivative of the likelihood function has a uniformly bounded moment, i.e.~$E\|{l}_t''\|_{K}<\infty$.
\end{lemma}

\begin{proof}
The second derivatives of the likelihood  is
$$l_t''(\pmb{\kappa})=
\begin{bmatrix}
   l_t^{\pmb \xi\pmb \xi}(\pmb{\kappa})        &   l_t^{{\pmb \xi} \lambda}(\pmb{\kappa})\ \tilde \lambda_t'(\pmb{\theta})^\top   \\
\tilde \lambda_t'(\pmb{\theta}) \ l_t^{{\pmb \xi} \lambda}(\pmb{\kappa})^\top       &  l_t^{\lambda \lambda}(\pmb{\kappa}) \ \tilde\lambda_t'(\pmb{\theta})\ \tilde\lambda_t'(\pmb{\theta})^\top +  l_t^\lambda (\pmb{\kappa}) \tilde\lambda_t''(\pmb{\theta})
\end{bmatrix},
$$
where $l_t^{\pmb \xi\pmb \xi}(\pmb{\kappa}) =\partial^2 l_t(\pmb{\kappa})/\partial \pmb{\xi} \partial \pmb{\xi}^\top$, $l_t^{\pmb \xi\lambda}(\pmb{\kappa}) =\partial^2 l_t(\pmb{\kappa})/\partial \pmb{\xi} \partial \lambda$, $l_t^{\lambda}(\pmb{\kappa}) =\partial l_t(\pmb{\kappa})/\partial \lambda$ and $l_t^{\lambda \lambda}(\pmb{\kappa}) =\partial^2 l_t(\pmb{\kappa})/\partial \lambda^2$.
By an application of Lemma \ref{lemma100}, we obtain that 
\begin{align*}
\mathbb{E}\|l_t''\|_{K}\le & \ \mathbb{E}\|l_t^{\pmb \xi \pmb \xi}\|_{K}+2 \  \mathbb{E}\|l_t^{{\pmb \xi} \lambda}\ \tilde \lambda_t'^\top\|_{K} + \mathbb{E} \| l_t^{\lambda \lambda} \ \tilde\lambda_t'\ \tilde\lambda_t'^\top\|_{K}+ \mathbb{E}\| l_t^\lambda \ \tilde\lambda_t''\|_{K}\\
\le & \ \mathbb{E}\|l_t^{\pmb \xi \pmb \xi}\|_{K}+2  (c_{\lambda \alpha}+c_{\lambda r}) \mathbb{E}\|\tilde \lambda_t'\|_{\Theta} + c_{\lambda\lambda}\ \mathbb{E} \|\tilde\lambda_t'\|_{\Theta}^2+c_\lambda \ \mathbb{E}\|\tilde \lambda_t''\|_{\Theta}.
\end{align*}
Therefore, $\mathbb{E}\|l_t''\|_{K}<\infty$ since $\mathbb{E} \|\tilde\lambda_t'\|_{\Theta}^2<\infty$ and $\mathbb{E}\|\tilde \lambda_t''\|_{\Theta}<\infty$ by Lemmas \ref{lemma_n1.1} and \ref{lemma_n1.2}, and $\mathbb{E}\|l_t^{\pmb \xi \pmb \xi}\|_{K}<\infty$  by Lemma \ref{lemma100},  because all entries of $l_t^{\pmb \xi \pmb \xi}$ are bounded by linear combinations of $y_t$ and $\tilde \lambda_t$, which have bounded moments.
\end{proof}

\begin{lemma}\label{lemma_n0}
Let the assumptions of Theorem \ref{th4} hold, then 
$$\sqrt{T} {L}_T'({\pmb{\kappa}}_0)\xrightarrow{d}N(\textbf{0}, \textbf{F}_0), \quad T\rightarrow \infty.$$
\end{lemma}

\begin{proof}
The  derivative of the  likelihood  function $l_t(\pmb\kappa)$ is 
$$l_t'(\pmb{\kappa})=
\begin{bmatrix}
    l_t^{\pmb \xi}(\pmb{\kappa})\vspace{0.15cm}\\
    l_t^\lambda(\pmb{\kappa})\tilde\lambda_t'(\pmb{\theta})
\end{bmatrix},
$$
where $ l_t^{\pmb \xi}(\pmb{\kappa}) = \partial l_t(\pmb\kappa)/\partial \pmb\xi$ and $ l_t^\lambda(\pmb{\kappa})=\partial l_t(\pmb\kappa)/\partial \tilde \lambda_t(\pmb\theta)$.
First, we obtain that the derivative of the likelihood has a uniformly bounded second moment, $\mathbb{E}\| {l}_{t}'\|_K^2<\infty$, as follows
\begin{align*}
\mathbb{E}\| {l}_{t}'\|_K^2 & \le \mathbb{E}[(\| l_t^{\pmb \xi} \|_K
+  c_\lambda   \| \tilde\lambda_t' \|_\Theta)^2]\\
 & \le \mathbb{E}\| l_t^{\pmb \xi} \|_K^2+ c_\lambda^2 \mathbb{E}\| \tilde\lambda_t' \|_\Theta^2+ 2c_\lambda (\mathbb{E}\| \tilde\lambda_t' \|_\Theta^2\mathbb{E}\| l_t^{\pmb \xi} \|_K^2)^{1/2}<\infty,
\end{align*} 
where the inequalities hold  by Lemma \ref{lemma100} together with the Cauchy-Schwartz inequality.
Second, we note that $\mathbb{E}[l_t'(\pmb{\kappa}_0)|\mathcal{F}_{t-1}]=0$ with probability 1. In particular, $\lambda_t'$ is $\mathcal{F}_{t-1}$ measurable and therefore 
$$\mathbb{E}[l_t'(\pmb{\kappa}_0)|\mathcal{F}_{t-1}]=
\begin{bmatrix}
     \mathbb{E}[l_t^{\pmb \xi}(\pmb{\kappa}_0)|\mathcal{F}_{t-1}]\vspace{0.15cm}\\
    \mathbb{E}[l_t^\lambda(\pmb{\kappa}_0)|\mathcal{F}_{t-1}] \  \lambda_t'
\end{bmatrix}.$$
Finally, $\mathbb{E}[l_t^{\pmb \xi}(\pmb{\kappa}_0)|\mathcal{F}_{t-1}]$ and $\mathbb{E}[l_t^\lambda(\pmb{\kappa}_0)|\mathcal{F}_{t-1}]$ are equal to zero a.s.~since they are the conditional scores of the BNB pmf evaluated at the true parameter vector  $(\lambda_t,r_0, \alpha_0)$.

Therefore, we have that  $\{{l}_{t}'(\pmb{\kappa}_0)\}_{t\in\mathbb{N}}$ is a martingale difference sequence with finite second moment. As a result, we conclude that 
$$\sqrt{T} {L}_T'({\pmb{\theta}}_0)\xrightarrow{d}N\left(\textbf{0}, \mathbb{E}[l_t'(\pmb{\kappa}_0)l_t'(\pmb{\kappa}_0)^\top]\right),\quad \text{as} \quad T\rightarrow \infty,$$ 
by an application of the Central Limit Theorem for martingale difference sequences, see \cite{billingsley}.

Finally, we note that the Fisher information matrix equality $\pmb{\text{F}}_0=-\mathbb{E}[l_t''(\pmb{\kappa}_0)]=\mathbb{E}[l_t'(\pmb{\kappa}_0)l_t'(\pmb{\kappa}_0)^\top]$ follows by standard arguments since $l_t(\pmb \kappa_0)$ is the true conditional log pmf evaluated at $y_t$ and the likelihood function is twice continuously differentiable with a uniformly bounded moment, which allow us to interchange integration with differentiation.
\end{proof}

\begin{lemma}\label{lemma_n3}
Let the assumptions of Theorem \ref{th4} hold, then the Fisher information matrix is positive definite, i.e.~$\pmb{\text{F}}_0=-\mathbb{E}[l_t''({\pmb{\kappa}}_0)]=\mathbb{E}[l_t'({\pmb{\kappa}}_0)l_t'({\pmb{\kappa}}_0)^\top]>0$.
\end{lemma}

\begin{proof}
	First, we note that obviously $\pmb{\text{F}}_0=\mathbb{E}[l_t'({\pmb{\kappa}}_0)l_t'({\pmb{\kappa}}_0)^\top]$ is positive semi-definite. Therefore, we only need to show that $\pmb{\text{F}}_0$ is non-singular, i.e.~$\pmb{\text{x}}^\top  \pmb{\text{F}}_0 \ \pmb{\text{x}} =0$  only if $\pmb{\text{x}}=\pmb{\text{0}}$. This is the equivalent of showing that $\pmb{\text{x}}^\top l_t'({\pmb{\kappa}}_0)=0$ a.s.~only if $\pmb{\text{x}}=\pmb{\text{0}}$. Consider the partition $\pmb{\text{x}}=(\pmb{\text{v}},\pmb{\text{w}})^\top$, where $\pmb{\text{v}}\in \mathbb{R}^2$ and $\pmb{\text{w}}\in \mathbb{R}^{k-2}$. We have that
$$\pmb{\text{x}}^\top l_t'({\pmb{\kappa}}_0)=\pmb{\text{v}}^\top\begin{bmatrix}
    \frac{\partial l_t({\pmb{\kappa}}_0)}{\partial r}\vspace{0.15cm}\\
    \frac{\partial l_t({\pmb{\kappa}}_0)}{\partial \alpha}
\end{bmatrix}+\pmb{\text{w}}^\top\frac{\partial l_t({\pmb{\kappa}}_0)}{\partial \pmb{\theta}}.$$
In the following, we show by contradiction that $\pmb{\text{x}}^\top l_t'({\pmb{\kappa}}_0)=0$ a.s.~implies $\pmb{\text{x}}=\pmb{\text{0}}$. There are 3 different cases.

$\pmb{\text{v}} \neq \pmb{0}, \pmb{\text{w}} = \pmb{0}$) This would mean that
$$\pmb{\text{v}}^\top\begin{bmatrix}
    \frac{\partial l_t({\pmb{\kappa}}_0)}{\partial r}\vspace{0.15cm}\\
    \frac{\partial l_t({\pmb{\kappa}}_0)}{\partial \alpha}
\end{bmatrix}= 0\quad \text{a.s.},$$
However, it is trivial to see that $ \frac{\partial l_t({\pmb{\kappa}}_0)}{\partial r}$ and $ \frac{\partial l_t({\pmb{\kappa}}_0)}{\partial \alpha}$ are linearly independent random variables and therefore the above equation cannot be true if $\pmb{\text{v}} \neq \pmb{0}$.

$\pmb{\text{v}} = \pmb{0}, \pmb{\text{w}}\neq \pmb{0}$) This would  mean that
$$\pmb{\text{w}}^\top \frac{\partial l_t({\pmb{\kappa}}_0)}{\partial \pmb{\theta}}  =\pmb{\text{w}}^\top \ \frac{\partial l_t({\pmb{\kappa}}_0)}{\partial \tilde\lambda_t(\pmb{\theta})} \frac{\partial \tilde\lambda_t(\pmb{\theta}_0)}{\partial \pmb{\theta}} = 0 \quad  \text{a.s.},$$
which would imply that $\pmb{\text{w}}^\top \frac{\partial \tilde \lambda_t(\pmb{\theta}_0)}{\partial \pmb{\theta}} = 0$. However, this cannot be true because  $\pmb{\text{w}}^\top\frac{\partial \tilde \lambda_t(\pmb{\theta}_0)}{\partial \pmb{\theta}} = 0$ a.s.~is ruled out by Assumption \ref{a5}.

$\pmb{\text{v}}\neq \pmb{0}, \pmb{\text{w}}\neq \pmb{0}$) This would imply that 
$$\pmb{\text{v}}^\top\begin{bmatrix}
    \frac{\partial l_t({\pmb{\kappa}}_0)}{\partial r}\vspace{0.15cm}\\
    \frac{\partial l_t({\pmb{\kappa}}_0)}{\partial \alpha}
\end{bmatrix}\left(\frac{\partial l_t({\pmb{\kappa}}_0)}{\partial \lambda_t }\right)^{-1} = \pmb{\text{w}}^\top \frac{\partial  \tilde \lambda_t(\pmb{\theta}_0)}{\partial \pmb{\theta}}  \quad  \text{a.s.}$$
However, this cannot be true because $\frac{\partial  \tilde \lambda_t}{\partial \pmb{\theta}}$ is measurable with respect to $\mathcal{F}_{t-1}$  and instead the left hand side of the above equation is not $\mathcal{F}_{t-1}$-measurable as it depends on $y_t$. This concludes the proof of the Lemma.
\end{proof}

\begin{lemma}\label{lemma_n1,5}
Let the assumptions of Theorem \ref{th4} hold, then 
$$\sqrt{T} \ \|\hat L_T' - L_T'\|_{K}\xrightarrow{a.s.}0, \quad \text{as} \quad T\rightarrow \infty.$$
\end{lemma}

\begin{proof}

By the mean value theorem and given the upper bounds in Lemma \ref{lemma100}, we obtain that
\begin{align*}
\|\hat l_t' -l_t'\|_K &\le\| \hat l_t^{\pmb \xi} - l_t^{\pmb \xi}\|_K + \| \hat l_t^\lambda \ \hat \lambda_t'- l_t^\lambda \ \tilde \lambda_t'\|_K \\
&\le\| \hat l_t^{\pmb \xi} - l_t^{\pmb \xi}\|_K + \|  l_t^\lambda (\hat \lambda_t'-\tilde \lambda_t')\|_K+ \| (\hat l_t^\lambda - l_t^\lambda) \hat \lambda_t'\|_K\\
&\le (c_{\lambda\alpha}+c_{\lambda r}+a_{\lambda \alpha}y_t)\| \hat \lambda_t  -\tilde \lambda_t\|_\Theta + c_\lambda \ \| \hat \lambda_t'-\tilde \lambda_t'\|_\Theta+ c_{\lambda\lambda}\ (1+\|\tilde \lambda_t'\|_\Theta) \| \hat \lambda_t  -\tilde \lambda_t\|_\Theta.
\end{align*}
Therefore, $\|\hat l_t' -l_t'\|_K\xrightarrow{e.a.s.}0$ by an application of Lemma 2.1 of \cite{SM2006} since   $\mathbb{E}\|\tilde\lambda_t'\|_{\Theta}<\infty$. 
Finally, we obtain  that 
$$\lim_{T\rightarrow \infty }T \ \|\hat L_T' - L_T'\|_{K} \le \sum_{i=1}^T\|\hat l_t' -l_t'\|_K <\infty \quad \text{a.s.}$$
This concludes the proof of the lemma.
\end{proof}

\subsection{Lemmas on the gamma function and the BNB pmf} 

\begin{lemma}\label{lemma1}
Let $X\sim\mathcal{BNB}(\lambda_1,r, \alpha)$ and $Y\sim\mathcal{BNB}(\lambda_2,r, \alpha)$ with $\lambda_1\ge \lambda_2$. Then  $X$ is stochastically greater than $Y$, $X\ge_{st}Y$,
$$F_{\lambda_1}(z)\le F_{\lambda_2}(z), \quad \text{for any}\; z \in \mathbb{N},$$
where $F_{\lambda_1}$ and $ F_{\lambda_2}$ denote the cumulative distribution functions of $X$ and $Y$, respectively.
\end{lemma}
\begin{proof}
The proof of this lemma is an immediate consequence of Theorem 4.2~of \cite{wang2011one} since likelihood ratio ordering implies stochastic ordering.
\end{proof}

\begin{lemma}\label{lemma_gamma}
Let $x_t$ be a random variable such that $x_t\ge\bar x$ with probability 1  for some constant $\bar x>0$. Furthermore, assume that $\mathbb{E}(x_t)<\infty$, then
$$\mathbb{E}|\log\Gamma(x_t+c) - \log\Gamma(x_t)|<\infty,$$
for any $c\in \mathbb{R}^+$.
\end{lemma}
\begin{proof}

First, we note that $\log\Gamma(x)$ is monotone increasing in $x\in (2,\infty)$. Therefore, for any $c>0$ and $x>2$, we have that 
\begin{align*}
|\log\Gamma(x+c) - \log\Gamma(x)| & \le  (c+1)\log(x)\\
& \le (c+1) x,
\end{align*}
where the first inequality  follows immediately from the well know recursive equation $\log\Gamma(x+1)-\log\Gamma(x)=\log(x)$, together with the monotonicity of the gamma function in $(2,\infty)$. Finally, denoting with $\pmb{1}_{A}(x)$ the indicator function of a subset $A\in \mathbb{R}$, we obtain that  
\begin{align}
\mathbb{E}\big|\log\Gamma(x_t+c) - \log\Gamma(x_t)|\le& \mathbb{E}\big|\big(\log\Gamma(x_t+c) - \log\Gamma(x_t)\big)\pmb{1}_{[\bar x, 2]}(x_t)\big|\nonumber \\
&+\mathbb{E}\big|\big(\log\Gamma(x_t+c) - \log\Gamma(x_t)\big)\pmb{1}_{(2, \infty)}(x_t)\big|\nonumber\\
\le& \sup_{x\in [\bar x, 2]} (|\log\Gamma(x+c)|+|\log\Gamma(x)|) \label{sup_gamma}\\
&+(c+1)\mathbb{E}(x_t+2) <\infty, \label{e_gamma}
\end{align}
where (\ref{sup_gamma}) is finite because of the continuity of $\log \Gamma$ in the compact set $[\bar x, 2]$ and (\ref{e_gamma}) is finite given that $\mathbb{E}(x_t)<\infty$. This concludes the proof.
\end{proof}

\begin{lemma}\label{lemma100}
The following inequalities are satisfied for any $\lambda\in [\bar c,\infty)$, $y \in \mathbb{N}$ and $\pmb\xi=(r,\alpha)^\top\in \Xi$, where   is a compact set  $\Xi \subset (0,\infty)\times (2,\infty)$.

\begin{enumerate}[(i)]
\item $\left| {\partial \log p(y|\lambda,r,\alpha)}/{\partial \lambda}\right| \le c_\lambda.$
\item $\left| {\partial \log p(y|\lambda,r,\alpha)}/{\partial r}\right| \le c_r + a_r \ (y+\lambda).$
\item $\left| {\partial \log p(y|\lambda,r,\alpha)}/{\partial \alpha}\right| \le c_\alpha+a_\alpha \ y.$
\item $\left| {\partial^2 \log p(y|\lambda,r,\alpha)}/{\partial \lambda^2}\right| \le c_{\lambda\lambda}.$
\item $\left| {\partial^2 \log p(y|\lambda,r,\alpha)}/{\partial r^2}\right| \le c_{rr}+ a_{rr} \ \lambda.$
\item $\left| {\partial^2 \log p(y|\lambda,r,\alpha)}/{\partial \alpha^2}\right| \le c_{\alpha\alpha} +a_{\alpha\alpha} \ y.$ 
\item $\left| {\partial^2 \log p(y|\lambda,r,\alpha)}/{\partial \lambda\partial r}\right| \le c_{\lambda r}.$
\item $\left| {\partial^2 \log p(y|\lambda,r,\alpha)}/{\partial \lambda\partial \alpha}\right| \le c_{\lambda\alpha}.$
\item $\left| {\partial^2 \log p(y|\lambda,r,\alpha)}/{\partial r\partial \alpha}\right| \le c_{r\alpha}+a_{r\alpha}\ \lambda.$
\end{enumerate}
for some positive constants $c_{i}$, $a_{i}$, $c_{ij}$ and  $a_{ij}$, with $i,j\in\{\lambda,r,\alpha\}$.

\end{lemma}

\begin{proof}

Fist we show that, for any $z>0$ and $c\ge0$, the following inequalities are satisfied 
\begin{align}
0 \le\psi(z+c)- \psi(z)\le \frac{c+1}{z},\label{in2}\\
   0 \le\psi_1(z)-\psi_1(z+c)\le \frac{c+1}{z^2},\label{in3}
\end{align}
where  $\psi$ and $\psi_1$ denote the  digamma and trigamma functions, respectively. 
In particular, it is well known that the digamma function satisfies the following recurrence equation
$ \psi(z+1)-\psi(z)= \frac{1}{z},$ for any $z>0$.
Therefore, since  $\psi(z)$ is a monotone increasing for $z>0$, we immediately obtain that (\ref{in2}) is satisfied.
Similarly, the  inequality  in (\ref{in3}) is obtained noticing and the trigamma function satisfies the recurrence  equation
$ \psi_1(z)-\psi_1(z+1)=  \frac{1}{z^2},$ for any $z>0$, and  $\psi_1(z)$  is monotone decreasing for $z>0$.

In the following, we rely on (\ref{in2}) and (\ref{in3}) to show that the inequalities (i)-(ix) are satisfied. Here, for simplicity of notation, we define $\gamma = (\alpha-1)/r$.
\begin{enumerate}[(i)]
\item We obtain that
\begin{align*}
\left|\frac{\partial \log p(y |\lambda ,r,\alpha)}{\partial \lambda }\right| \le & \gamma  \Big(\psi(\gamma\lambda +y +\alpha +r) -\psi(\gamma\lambda +y)+    \psi(\gamma\lambda  +\alpha)-\psi(\gamma\lambda)\Big)\\
\le&    \frac{\gamma(\alpha+r+1)}{\gamma\lambda +  y } +\frac{\alpha+1}{\lambda }\le   \frac{  2\alpha +2+ r}{\bar c}\le c_\lambda,
\end{align*}
where the first inequality follows by standard differentiation together with the triangle inequality, the second by   (\ref{in2}), the third by taking the supremum over $\lambda \in [\bar c,\infty)$ and $y \in \mathbb{N}$, and the last by taking the supremum over $\pmb{\xi}$ in a compact set $\Xi$.
 
\item  Following similar steps as in (i), we obtain that
\begin{align*}
\left|\frac{\partial \log p(y|\lambda,r,\alpha)}{\partial r}\right|  \le &    \psi(y+r)-\psi(r)  + |\psi(\alpha+r)|+ |\psi( \gamma\lambda+y+\alpha +r)| \\
& +\frac{\gamma\lambda}{r} \Big(\psi(\gamma\lambda+y+\alpha +r)-\psi(\gamma\lambda+y)  + \psi(\gamma\lambda+\alpha)-\psi(\gamma\lambda) \Big)\\
\le & \frac{y+1}{r}+ |\psi(\alpha+r)| + \gamma\lambda+y+\alpha +r+ 1    + \frac{(\alpha+r+1) \gamma \lambda}{r(\gamma \lambda +y)} + \frac{\alpha+1}{r}\\
\le &  c_r+ a_r (y+\lambda),
\end{align*}
where the first inequality follows by standard differentiation together with the triangle inequality, the second by   (\ref{in2}) and noticing that $|\psi(z)| < z+1$ for $z>1$, the third by taking the supremum over $\lambda \in [\bar c,\infty)$, $y \in \mathbb{N}$ and $\pmb{\xi}\in\Xi$.

\item Similarly as before, we have 
\begin{align*}
\left|\frac{\partial \log p(y|\lambda,r,\alpha)}{\partial \alpha}\right|&\le   \left(1+\frac{\lambda}{r}\right)\Big(\psi(\gamma\lambda +\alpha +y+r)-\psi(\gamma\lambda+\alpha)\Big)\\
&\quad  +\frac{\lambda}{r} \Big(\psi (\gamma\lambda+y )-\psi (\gamma\lambda )\Big)+|\psi(\alpha+r)|+|\psi(\alpha)|\\
&\le\frac{(y+r+1)(\lambda+r)}{r(\gamma\lambda+\alpha)}+  \frac{y+1}{r \gamma}+|\psi(\alpha+r)|+|\psi(\alpha)|\le c_\alpha+ a_\alpha y,
\end{align*}
where the first inequality follows by standard differentiation together with the triangle inequality, the second by an application of (\ref{in2}) and the last by taking the supremum over $\lambda \in [\bar c,\infty)$ and $\pmb{\xi}\in\Xi$.

\item We obtain that
 \begin{align*}
\left|\frac{\partial^2 \log p(y|\lambda,r,\alpha)}{\partial \lambda^2}\right| \le& \gamma^2  \Big(\psi_1(\gamma\lambda+y)-\psi_1(\gamma\lambda +y+\alpha+r)   +  \psi_1(\gamma\lambda)-\psi_1( \gamma\lambda+\alpha) \Big)\\
 \le &     \frac{\gamma^2 (\alpha+r+1)}{( \gamma\lambda+y)^2} +\frac{\alpha+1}{\lambda^2}
\le     \frac{2\alpha+2+r}{\bar c^2}\le c_{\lambda\lambda},
\end{align*}
where the first inequality follows by standard differentiation together with the triangle inequality, the second by (\ref{in3}) and the last two by taking the supremum over $\lambda\in [\bar c,\infty)$, $\pmb\xi\in \Xi$ and $y\in\mathbb{N}$.

\item  We have that
\begin{align*}
\left|\frac{\partial^2 \log p(y |\lambda ,r,\alpha)}{\partial r^2}\right|\le&  \psi_1(y+r)+\psi_1(r) + \psi_1(\alpha+r)  +\frac{2\gamma\lambda+r}{r}\psi_1(\gamma\lambda+y+\alpha +r)\\
& +\frac{\gamma^2\lambda^2}{r^2} \Big(\psi_1(\gamma\lambda+y)-\psi_1(\gamma\lambda+y+\alpha +r)  + \psi_1(\gamma\lambda)-\psi_1(\gamma\lambda +\alpha) \Big)\\
& +\frac{2\gamma\lambda}{r^2} \Big(\psi(\gamma\lambda+y+\alpha +r)-\psi(\gamma\lambda+y)   +  \psi(\gamma\lambda+\alpha)-\psi(\gamma\lambda) \Big)\\
\le&    \frac{2\gamma\lambda+r}{r}\psi_1 (r) +\frac{\gamma^2\lambda^2(\alpha+r+1)}{r^2(\gamma \lambda+y)^2} + \frac{3(\alpha+1)}{r^2} + \frac{2\gamma\lambda (\alpha+r+1)}{r^2(\gamma\lambda+y)}  \\
\le  &  \frac{2\gamma\lambda+r}{r}\psi_1 (r)+\frac{3(r+2\alpha+2)}{r^2}\le c_{rr} +a_{rr}\ \lambda,
\end{align*}
where the first inequality follows by standard differentiation, the second by (\ref{in2})  and (\ref{in3}) together with the fact that $\psi_1(x)$, $x>0$, is strictly positive and monotone decreasing and the last two inequality follow by taking the supremum over $\lambda\in [\bar c,\infty)$, $\pmb\xi\in \Xi$ and $y\in\mathbb{N}$.

\item We obtain that
\begin{align*}
\left|\frac{\partial^2 \log p(y|\lambda,r,\alpha)}{\partial \alpha^2}\right|&\le   \left(1+\frac{\lambda}{r}\right)^2\Big(\psi_1(\gamma\lambda+\alpha)-\psi_1(\gamma\lambda +\alpha +y+r)\Big)\\
&\quad  +\frac{\lambda^2}{r^2} \Big(\psi_1(\gamma\lambda)-\psi_1(\gamma\lambda+y)\Big)+\psi_1(\alpha+r)+\psi_1(\alpha)\\
&\le  \frac{(r+ \lambda)^2(y+r+1)}{r^2(\gamma\lambda+\alpha)^2}    + \frac{y+1}{r^2\gamma^2} +2\psi_1(\alpha)\le c_{\alpha\alpha}+a_{\alpha\alpha}\ y,
\end{align*}
where the first inequality follows by standard differentiation together with the triangle inequality, the second by (\ref{in3}) and the third by taking the supremum over $\lambda\in [\bar c,\infty)$ and $\pmb\xi\in \Xi$.

\item As before, we obtain
\begin{align*}
\left|\frac{\partial^2 \log p(y|\lambda,r,\alpha)}{\partial \lambda\partial r}\right|\le &\gamma \psi_1 (\gamma\lambda +y +\alpha +r ) \\
&+ \frac{\gamma^2\lambda}{r}  \Big( \psi_1( \gamma\lambda+y)-\psi_1(\gamma\lambda+y +\alpha +r)
+ \psi_1(\gamma\lambda)-  \psi_1(\gamma\lambda +\alpha)\Big)\\
&+ \frac{\gamma}{r}  \Big(\psi(\gamma\lambda+y +\alpha +r)-\psi(\gamma\lambda+y )  +  \psi(\gamma\lambda+\alpha)-\psi(\gamma\lambda)\Big)\\
\le &\gamma\psi_1(\alpha)+ \frac{\gamma^2\lambda (\alpha+r+1)}{r(\gamma\lambda+y)^2}+   \frac{2(\alpha+1)}{r\lambda}+    \frac{\gamma(\alpha+r+1)}{r(\gamma\lambda+y)}\le  c_{\lambda r},
\end{align*}
where the first inequality follows by standard differentiation together with the triangle inequality, the second by  (\ref{in2}) and (\ref{in3}) and the third by taking the supremum over $\lambda\in [\bar c,\infty)$, $y\in \mathbb{N}$ and $\pmb\xi\in \Xi$.

\item  We have that
\begin{align*}
\left|\frac{\partial^2 \log p(y|\lambda,r,\alpha)}{\partial \lambda\partial \alpha}\right| \le & \gamma \psi_1 (\gamma\lambda+y +\alpha +r )+ \gamma \psi_1 (\gamma\lambda +\alpha)\\
&+ \frac{\gamma\lambda}{r}   \Big(\psi_1(\gamma\lambda+y)-\psi_1 (\gamma\lambda +y +\alpha +r ) +      \psi_1(\gamma\lambda)-\psi_1(\gamma\lambda +\alpha)\Big)\\
&+     \frac{1}{r}  \Big( \psi (\gamma\lambda +y +\alpha +r)-\psi(\gamma\lambda +y)+   \psi (\gamma\lambda +\alpha)-\psi (\gamma\lambda) \Big)\\
\le & 2\gamma \psi_1(\alpha) + \frac{\gamma\lambda(\alpha+r+1)}{r(\gamma\lambda+y)^2}+      \frac{2(\alpha+1)}{r\gamma\lambda} + \frac{\alpha+r+1}{r(\gamma\lambda+y)}\le c_{\lambda\alpha},
\end{align*}
where the first inequality follows by standard differentiation together with the triangle inequality, the second by  (\ref{in2}) and (\ref{in3}) and the third by taking the supremum over $\lambda\in [\bar c,\infty)$, $y\in \mathbb{N}$ and $\pmb\xi\in \Xi$.

\item As before, we obtain
\begin{align*}
\left|\frac{\partial^2 \log p(y|\lambda,r,\alpha)}{\partial r\partial \alpha}\right|\le & \psi_1(\alpha+r) + \frac{\gamma\lambda}{r}\psi_1(\gamma\lambda+\alpha )  + \frac{r+(\gamma+1)\lambda}{r}\psi_1(\gamma\lambda+y+\alpha +r) \\
& +\frac{\lambda}{r^2}\Big(   \psi(\gamma\lambda+y+\alpha +r)-\psi(\gamma\lambda+y)   + \psi(\gamma\lambda+\alpha)-\psi (\gamma\lambda)\Big)\\
& +\frac{\gamma\lambda^2}{r^2}  \Big(\psi_1 (\gamma\lambda+y)-\psi_1(\gamma\lambda+y+\alpha +r)  +  \psi_1 (\gamma\lambda )-\psi_1 (\gamma\lambda+\alpha) \Big)\\
\le & \frac{2r+(2\gamma+1)\lambda}{r}\psi_1(\alpha) +\frac{\lambda(\alpha+r+1)}{r^2(\gamma\lambda+y)}+\frac{2(\alpha+1)}{r^2\gamma}+\frac{\gamma\lambda^2 (\alpha+r+1)}{r^2(\gamma\lambda+y)^2}\\
\le&  c_{r\alpha}+a_{r\alpha}\lambda,
\end{align*}
where the first inequality follows by standard differentiation together with the triangle inequality, the second by  (\ref{in2}) and (\ref{in3}) and the third by taking the supremum over $\lambda\in [\bar c,\infty)$, $y\in \mathbb{N}$ and $\pmb\xi\in \Xi$.
\end{enumerate}
\end{proof}

\bibliographystyle{apalike}
\bibliography{references}

\begin{thebibliography}{}

\bibitem[Billingsley, 1999]{billingsley}
Billingsley, P. (1999).
\newblock {\em Convergence of probability measures}.
\newblock Wiley, New York, 2nd edition.

\bibitem[Blasques et~al., 2018]{blasques2018feasible}
Blasques, F., Gorgi, P., Koopman, S.~J., Wintenberger, O., et~al. (2018).
\newblock Feasible invertibility conditions and maximum likelihood estimation
  for observation-driven models.
\newblock {\em Electronic Journal of Statistics}, 12(1):1019--1052.

\bibitem[Bougerol, 1993]{Bougerol1993}
Bougerol, P. (1993).
\newblock Kalman filtering with random coefficients and contractions.
\newblock {\em SIAM Journal on Control and Optimization}, 31(4):942--959.

\bibitem[Creal et~al., 2011]{creal2011dynamic}
Creal, D., Koopman, S.~J., and Lucas, A. (2011).
\newblock A dynamic multivariate heavy-tailed model for time-varying
  volatilities and correlations.
\newblock {\em Journal of Business \& Economic Statistics}, 29(4):552--563.

\bibitem[Creal et~al., 2013]{Creal2013}
Creal, D., Koopman, S.~J., and Lucas, A. (2013).
\newblock Generalized autoregressive score models with applications.
\newblock {\em Journal of Applied Econometrics}, 28(5):777--795.

\bibitem[Davis et~al., 2003]{davis2003observation}
Davis, R.~A., Dunsmuir, W.~T., and Streett, S.~B. (2003).
\newblock Observation-driven models for poisson counts.
\newblock {\em Biometrika}, 90(4):777--790.

\bibitem[Davis et~al., 2016]{davis2016handbook}
Davis, R.~A., Holan, S.~H., Lund, R., and Ravishanker, N. (2016).
\newblock {\em Handbook of discrete-valued time series}.
\newblock CRC Press.

\bibitem[Davis and Liu, 2016]{davis2016theory}
Davis, R.~A. and Liu, H. (2016).
\newblock Theory and inference for a class of nonlinear models with application
  to time series of counts.
\newblock {\em Statistica Sinica}, pages 1673--1707.

\bibitem[Davis and Wu, 2009]{davis2009negative}
Davis, R.~A. and Wu, R. (2009).
\newblock A negative binomial model for time series of counts.
\newblock {\em Biometrika}, 96(3):735--749.

\bibitem[Doukhan and Wintenberger, 2008]{doukhan2008weakly}
Doukhan, P. and Wintenberger, O. (2008).
\newblock Weakly dependent chains with infinite memory.
\newblock {\em Stochastic Processes and their Applications},
  118(11):1997--2013.

\bibitem[Ferland et~al., 2006]{ferland2006integer}
Ferland, R., Latour, A., and Oraichi, D. (2006).
\newblock Integer-valued garch process.
\newblock {\em Journal of Time Series Analysis}, 27(6):923--942.

\bibitem[Fokianos et~al., 2009]{fokianos2009poisson}
Fokianos, K., Rahbek, A., and Tj{\o}stheim, D. (2009).
\newblock Poisson autoregression.
\newblock {\em Journal of the American Statistical Association},
  104(488):1430--1439.

\bibitem[Fox, 1972]{fox1972outliers}
Fox, A.~J. (1972).
\newblock Outliers in time series.
\newblock {\em Journal of the Royal Statistical Society. Series B
  (Methodological)}, pages 350--363.

\bibitem[Gorgi, 2018]{gorgi2018integer}
Gorgi, P. (2018).
\newblock Integer-valued autoregressive models with survival probability driven
  by a stochastic recurrence equation.
\newblock {\em Journal of Time Series Analysis}, 39(2):150--171.

\bibitem[Harvey, 2013]{H2013}
Harvey, A. (2013).
\newblock {\em Dynamic Models for Volatility and Heavy Tails: With Applications
  to Financial and Economic Time Series}.
\newblock New York: Cambridge University Press.

\bibitem[Harvey and Luati, 2014]{HL2014}
Harvey, A. and Luati, A. (2014).
\newblock Filtering with heavy tails.
\newblock {\em Journal of the American Statistical Association},
  109(507):1112--1122.

\bibitem[Opschoor et~al., 2017]{opschoor2017new}
Opschoor, A., Janus, P., Lucas, A., and Van~Dijk, D. (2017).
\newblock New heavy models for fat-tailed realized covariances and returns.
\newblock {\em Journal of Business \& Economic Statistics}, pages 1--15.

\bibitem[Rao, 1962]{rao1962}
Rao, R.~R. (1962).
\newblock Relations between weak and uniform convergence of measures with
  applications.
\newblock {\em The Annals of Mathematical Statistics}, 33(2):659--680.

\bibitem[Straumann and Mikosch, 2006]{SM2006}
Straumann, D. and Mikosch, T. (2006).
\newblock Quasi-maximum-likelihood estimation in conditionally heteroscedastic
  time series: A stochastic recurrence equations approach.
\newblock {\em The Annals of Statistics}, 34(5):2449--2495.

\bibitem[Wald, 1949]{wald1949}
Wald, A. (1949).
\newblock Note on the consistency of the maximum likelihood estimate.
\newblock {\em The Annals of Mathematical Statistics}, 20(4):595--601.

\bibitem[Wang, 2011]{wang2011one}
Wang, Z. (2011).
\newblock One mixed negative binomial distribution with application.
\newblock {\em Journal of Statistical Planning and Inference},
  141(3):1153--1160.

\bibitem[Zhu, 2011]{zhu2011negative}
Zhu, F. (2011).
\newblock {A negative binomial integer-valued GARCH model}.
\newblock {\em Journal of Time Series Analysis}, 32(1):54--67.

\end{thebibliography}

\end{document}